\documentclass[journal,twoside,web]{ieeecolor}

\usepackage{generic}
\usepackage{amsmath,amssymb,amsfonts}
\usepackage{mathtools}
\usepackage{graphicx}
\usepackage{stfloats}
\usepackage{color}
\usepackage{booktabs}
\usepackage{multirow}
\usepackage{algorithm}
\usepackage{algpseudocode}
\usepackage{url}
\usepackage{etoolbox}

\makeatletter
\@ifundefined{color@begingroup}%
  {\let\color@begingroup\relax
   \let\color@endgroup\relax}{}%

\def\fix@ieeecolor@hbox#1{%
  \hbox{\color@begingroup#1\color@endgroup}}

\patchcmd{\@makecaption}
  {\hbox}
  {\fix@ieeecolor@hbox}
  {}
  {\FAILED}

\patchcmd{\@makecaption}
  {\hbox}
  {\fix@ieeecolor@hbox}
  {}
  {\FAILED}
\makeatother

\newtheorem{assumption}{\bf Assumption}
\newtheorem{proposition}{\bf Proposition}
\newtheorem{theorem}{\bf Theorem}
\newtheorem{remark}{\bf Remark}

\newtheorem{lemma}{\bf Lemma}
\newtheorem{corollary}{\bf Corollary}

\begin{document}

\title{Reputation-Based Byzantine-Resilient Vector Consensus without Prespecified Fault Bounds}

\author{
Rui Huang, Changxin Liu,~\IEEEmembership{Member,~IEEE}, and Yang Shi,~\IEEEmembership{Fellow,~IEEE}
\thanks{Rui Huang is with the School of Aeronautical, Automotive, Chemical and Materials Engineering, Loughborough University, Loughborough, Leicestershire, LE11 3TU, UK (e-mail: {\tt\small r.huang@lboro.ac.uk}).
Changxin Liu is with the Key Laboratory of Smart Manufacturing in Energy Chemical Process, Ministry of Education, East China University of Science and Technology, Shanghai 200237, China (e-mail: {\tt\small changxinl@ecust.edu.cn}).
Yang Shi is with the Department of Mechanical Engineering, University of Victoria, Victoria, BC V8W 3P6, Canada (e-mail: {\tt\small yshi@uvic.ca}).
}
}

\maketitle

\begin{abstract}
    This paper proposes an adaptive reputation-based algorithm for Byzantine-resilient vector consensus in distributed multi-agent systems.
    The proposed algorithm does not require a prespecified bound on the number of Byzantine neighbors as an algorithmic input.
    Instead, each agent evaluates neighboring behavior from observed state discrepancies and updates reputation weights using accumulated behavioral evidence.
    The reputation mechanism adaptively redistributes communication weights and can assign exactly zero weight to sufficiently inconsistent neighbors.
    These reputation weights are incorporated into the local consensus updates to suppress adversarial influence while preserving reliable neighbor interactions.
    The resulting algorithm couples reputation adaptation with consensus dynamics while supporting high-dimensional vector states with low per-agent computational complexity.
    Extensive experiments on distributed systems against classical resilient consensus methods demonstrate improved Byzantine detection accuracy and more reliable consensus performance across different attack scenarios and state dimensions.
\end{abstract}


\section{Introduction}
\label{sec:introduction}

Over the past decade, advances in communication technologies and low-cost embedded devices have driven a shift from centralized architectures to distributed, networked systems, where computation and decision-making are carried out through local interactions among agents \cite{olfati2007consensus}.
A fundamental task in such systems is consensus, in which agents iteratively exchange information with neighbors to reach agreement on a common value, enabling scalability, reduced communication overhead, and robustness to failures in applications such as multi-robot coordination \cite{jadbabaie2003coordination}, sensor networks \cite{yu2009distributed}, distributed optimization \cite{nedic2009distributed}, and distributed game-theoretic decision making \cite{feng2025quantized}.

However, in real-world cyber-physical systems, agents may be subject to faults, attacks, or malicious compromises \cite{lamport1982byzantine,liu2024survey}, where misbehaving nodes can manipulate states or transmit deceptive information, potentially leading to unsafe system behavior in applications such as robotic swarms \cite{abbas2017improving,park2017fault}, distributed optimization \cite{su2016multi,sundaram2015consensus}, and learning \cite{liu2026byzantine}.
This challenge has led to the study of resilient consensus, which seeks to guarantee that all honest agents asymptotically agree despite adversarial nodes with unknown identities, while preserving safety or validity, i.e., ensuring that their states remain within a region determined by the initial honest states (e.g., the convex hull or interval).
A key difficulty in resilient consensus lies in the fact that honest agents cannot distinguish adversarial neighbors from honest ones, making it nontrivial to design update rules that simultaneously ensure agreement and safety.

Related work on resilient consensus spans trimming-based, geometric, and adaptive weighting approaches.
Recent surveys provide comprehensive overviews of resilient consensus and graph-theoretic approaches for distributed systems under adversarial conditions \cite{pirani2023graph}.
Early approximate agreement algorithms for real-valued states relied on iteratively removing extreme values to guarantee convergence under Byzantine faults \cite{dolev1986reaching}.
This idea was later extended to general directed networks through the W-MSR algorithm, which introduced the notion of network robustness to characterize the topological conditions required for local, fault-tolerant consensus \cite{leblanc2013resilient,tseng2013iterative}.
MSR-based resilient consensus has also been extended to event-triggered communication to reduce communication overhead \cite{wang2019resilient}.
Resilient consensus becomes substantially more challenging when agent states are vectors.
Early studies of Byzantine vector consensus established fundamental feasibility and resilience conditions for multidimensional agreement in complete networks \cite{vaidya2013byzantine,mendes2015multidimensional}, and were subsequently extended to incomplete communication graphs \cite{vaidya2014iterative}.
Unlike scalar consensus, independently applying scalar resilient updates to individual coordinates does not generally preserve the required convex-hull validity in the vector space.
To address this issue, geometric approaches construct safe points that remain within the convex hull determined by honest states.
Tverberg-partition-based constructions have been employed for fault-tolerant vector consensus and multi-robot rendezvous \cite{park2017fault}, while centerpoint-based methods provide an alternative geometric characterization with improved resilience guarantees \cite{abbas2022resilient}.
More recently, reputation-based mechanisms have been explored to dynamically downweight suspicious neighbors and achieve resilience under both synchronous and asynchronous communications \cite{ramos2023discrete}.
Although the analysis in \cite{ramos2023discrete} focuses on scalar-valued states, reputation-based weighting provides a natural mechanism for assigning scalar reliability scores to neighboring agents.
Related trust-based approaches exploit auxiliary stochastic observations to classify neighboring agents and construct trusted neighborhoods for resilient consensus \cite{akgun2025multi}.
Such approaches, however, rely on side information whose statistical properties distinguish legitimate from malicious communications.
Reinforcement-learning-based approaches have also been explored to adapt neighbor weights from interaction data, where rewards constructed from relative state differences are used to estimate neighbor credibility over time \cite{hou2021reinforcement,hou2023resilient}.
Recent works in federated learning have further investigated robust aggregation rules such as coordinate-wise median \cite{cambus2025coordinate} to mitigate Byzantine effects in high-dimensional settings.
Such federated-learning approaches commonly rely on centralized or all-to-all communication architectures, where global aggregation is available at each iteration.

Despite these advances, existing resilient consensus methods still face several limitations that restrict their applicability to general distributed systems:

\begin{enumerate}
    \item
          \textbf{Scalability to high-dimensional states.}
          Many classical resilient consensus methods are primarily designed for scalar states.
          Trimming-based approaches such as W-MSR do not directly extend to vector-valued states without additional geometric considerations, while coordinate-wise filtering may fail to preserve the desired geometric consistency.
          Geometric methods based on safe points or centerpoints provide principled extensions to vector states \cite{abbas2022resilient}, but their computational cost and fault-tolerance requirements generally become less favorable as the state dimension increases.
          This can limit their applicability to high-dimensional distributed tasks.

    \item
          \textbf{Dependence on structural assumptions and prespecified fault bounds.}
          Existing resilient consensus methods often rely on strong conditions on the communication graph or prior knowledge of the adversary model.
          Trimming-based methods typically require sufficient network redundancy, commonly characterized through robustness or degree conditions.
          Moreover, some filtering and adaptive-weighting schemes require a
          prespecified upper bound $f$ on the number of Byzantine neighbors
          \cite{ramos2023discrete}.
          Such prior knowledge may be unavailable in practice and can lead to
          unnecessarily conservative filtering when the prescribed bound exceeds
          the actual number of adversaries.
          A related tension has recently been established in Byzantine-robust distributed learning, where increasing the prescribed fault tolerance was shown to inevitably degrade the worst-case aggregation accuracy even in the absence of Byzantine workers \cite{yang2025tension}.

    \item
          \textbf{Limited use of behavioral history.}
          Many resilient consensus methods determine neighbor influence mainly from current observations or predefined filtering rules.
          Even adaptive-weighting approaches often update scalar trust values through prescribed mechanisms without explicitly exploiting accumulated behavioral consistency over time.
          As a result, persistent differences between reliable and unreliable communication behaviors may not be fully utilized when adjusting neighbor influence.
\end{enumerate}

These limitations motivate the need for a more adaptive and scalable algorithm that can handle high-dimensional states, reduce reliance on strong prior assumptions, and mitigate adversarial influence through behavior-driven reputation adaptation.

To address these limitations, we propose an adaptive reputation-based algorithm for Byzantine-resilient vector consensus.
Rather than relying on fault-bound-dependent filtering or computationally intensive geometric aggregation, the proposed method adapts neighbor influence directly from accumulated behavioral evidence.
Each agent evaluates neighboring behavior using an outlier-robust discrepancy measure and accumulates the resulting evidence over time.
The accumulated behavioral information is then used to adaptively update neighbor reputations, allowing sufficiently inconsistent neighbors to receive zero communication weight.
These reputations are incorporated into the local consensus updates to attenuate unreliable influences while preserving interactions with behaviorally consistent neighbors.
As the consensus process evolves, the reputation mechanism continuously adjusts communication weights according to the observed behavior of neighboring agents.
The resulting algorithm tightly couples reputation adaptation with consensus dynamics while supporting high-dimensional vector states with low per-agent computational complexity.

The main contributions of this paper are threefold:
\begin{enumerate}
    \item
          \textbf{Dimension-robust vector consensus with low complexity.}
          The proposed method operates directly on vector-valued states and avoids computationally expensive geometric constructions such as centerpoint computation.
          Neighbor assessment is reduced to scalar reputation updates derived from vector discrepancies, keeping the additional per-agent computation lightweight as the state dimension increases.
          This enables scalable resilient consensus over high-dimensional states while operating on the original communication graph.

    \item
          \textbf{No prespecified Byzantine fault bound.}
          We develop a reputation-based consensus mechanism that does not require a prespecified upper bound on the number of Byzantine neighbors as an algorithmic input.
          Instead of determining resilient updates from a known fault budget, each agent adjusts neighbor influence directly from observed behavioral evidence.
          This reduces reliance on prior adversary information while retaining a fully distributed implementation.

    \item
          \textbf{Reputation-based Byzantine behavior discrimination.}
          In addition to regulating neighbor influence during consensus, the observed reputations provide an interpretable measure of neighbor reliability.
          Persistent behavioral inconsistencies are reflected by reduced reputation values, enabling honest agents to distinguish unreliable neighbors from behaviorally consistent ones.
          The resulting reputation information can further support higher-level actions such as anomaly monitoring, neighbor isolation, or communication reconfiguration.
\end{enumerate}

The rest of the paper is organized as follows.
Section~\ref{sec:problem_formulation} formalizes the problem of Byzantine-resilient consensus in distributed networks.
Section~\ref{sec:proposed_algorithm} introduces the proposed algorithm, detailing its design and implementation.
Section~\ref{sec:analysis} establishes the convergence and resilience properties of the proposed algorithm under Byzantine attacks.
Section~\ref{sec:experiments} presents a range of distributed systems experiments, demonstrating the superior performance of the proposed method over classical resilient consensus approaches.
Finally, Section~\ref{sec:conclusion} concludes the paper and discusses potential future research directions.

\section{Preliminaries and Problem Formulation}
\label{sec:problem_formulation}

\subsection{Notation}

We use $\mathbb{R}$ to denote the set of real numbers, $\mathbb{R}^d$ to denote the $d$-dimensional Euclidean space, and $\mathbb{R}^{n \times d}$ to denote the set of $n \times d$ real matrices.
For a vector $x \in \mathbb{R}^d$, we denote its $k$-th entry by $(x)_k$.
For a set $S$, we denote its cardinality by $|S|$.
For a vector $x \in \mathbb{R}^d$, we denote its $\ell_p$-norm by $\|x\|_p$ for $p \ge 1$, and in particular $\|x\| = \|x\|_2$.
For a matrix $A \in \mathbb{R}^{n \times d}$, we denote its Frobenius norm by $\|A\|_F$ and its induced norms by $\|A\|_p$.
Similar to the vector case, we denote $\|A\| = \|A\|_2$.
For a set $S$ and a point $x$, we denote the infinity distance from $x$ to $S$ by $\operatorname{dist}_\infty(x,S) = \inf_{y \in S} \|x - y\|_\infty$.
We denote $I_d$ as the $d \times d$ identity matrix, and $\mathbf{1}_d$ as the $d$-dimensional vector of all ones.
If the dimension is clear from the context, we simply write $I$ and $\mathbf{1}$.
The indicator function is denoted by $\mathbb{I}\{\cdot\}$, which equals $1$ if the condition is true and $0$ otherwise.

\subsection{Consensus in the Presence of Malicious Nodes}

We consider a network of $n$ nodes represented by an undirected connected graph
$\mathcal{G}=(\mathcal{V},\mathcal{E})$, where
$\mathcal{V}=\{1,2,\dots,n\}$ and an edge $(i,j)\in\mathcal{E}$ indicates that nodes $i$ and $j$ can directly exchange information.
Each node $i$ maintains a state $x_i^{(t)}\in\mathbb{R}^d$ at discrete time $t$, and its neighbor set is denoted by
$\mathcal{N}_i=\{j\in\mathcal{V}\mid (i,j)\in\mathcal{E}\}$.

In the standard distributed consensus problem, the goal is to design a local update rule such that
\begin{equation*}
    \lim_{t\to\infty}\|x_i^{(t)}-x_j^{(t)}\|=0,
    \quad \forall i,j\in\mathcal{V},
\end{equation*}
using only information from neighboring nodes.
A typical consensus update is given by \cite{olfati2007consensus}
\begin{equation*}
    x_i^{(t+1)}
    =
    (1-\alpha)x_i^{(t)}
    +
    \alpha \frac{1}{|\mathcal{N}_i|}
    \sum_{j\in\mathcal{N}_i} x_j^{(t)},
\end{equation*}
where $\alpha\in(0,1]$ controls the influence of neighboring states.
This update forms a convex combination of node $i$'s current state and the average state of its neighbors, and achieves consensus under standard connectivity.

However, such averaging is vulnerable to Byzantine behavior.
Malicious nodes may send arbitrary or inconsistent messages to their neighbors, thereby distorting local averages and preventing honest nodes from reaching consensus.
We therefore partition the node set as $\mathcal{V}=\mathcal{H}\cup\mathcal{B}$, where $\mathcal{H}$ and $\mathcal{B}$ denote the honest and malicious nodes, respectively.
Honest nodes follow the prescribed protocol, while malicious nodes may behave arbitrarily.
The resilient consensus objective is to ensure
\begin{equation*}
    \lim_{t\to\infty}\|x_i^{(t)}-x_j^{(t)}\|=0,
    \quad \forall i,j\in\mathcal{H}.
\end{equation*}

\subsection{Attacker Model}

We consider a Byzantine adversarial model on a fixed communication graph.
Let $\mathcal{B}\subseteq\mathcal{V}$ denote the set of Byzantine nodes and $\mathcal{H}=\mathcal{V}\setminus\mathcal{B}$ the set of honest nodes.
Honest nodes follow the prescribed consensus protocol, whereas Byzantine nodes may arbitrarily update their internal states and send arbitrary, possibly inconsistent, values to different neighbors.
We focus on value-based attacks and assume reliable communication links, so that each node receives one message from each neighbor at every iteration.

We impose the following assumptions for resilience.

\begin{assumption}
    \label{assump:majority_honest}
    Each honest node has a strict majority of honest neighbors:
    \begin{equation*}
        |\mathcal{N}_i \cap \mathcal{H}| > |\mathcal{N}_i \cap \mathcal{B}|, \quad \forall i \in \mathcal{H}.
    \end{equation*}
\end{assumption}

\begin{assumption}
    \label{assump:honest_connected}
    The subgraph induced by $\mathcal{H}$ is connected.
\end{assumption}

For later analysis, define the stacked honest state as
\begin{equation*}
    X_{\mathcal H}^{(t)}
    \coloneqq
    (x_i^{(t)})_{i\in\mathcal H}.
\end{equation*}
Since a Byzantine node may transmit different values to different neighbors, let
\begin{equation*}
    \mathcal E_{\mathcal B\to\mathcal H}
    \coloneqq
    \left\{
    (j,i)
    :
    i\in\mathcal H,\
    j\in\mathcal N_i\cap\mathcal B
    \right\}
\end{equation*}
denote the set of Byzantine-to-honest communication links.
Let
\begin{equation*}
    m_{\mathcal B\to\mathcal H}
    \coloneqq
    |\mathcal E_{\mathcal B\to\mathcal H}|
\end{equation*}
denote the number of Byzantine-to-honest communication links.
Define the stacked Byzantine messages as
\begin{equation*}
    Z_{\mathcal B}^{(t)}
    \coloneqq
    \left(
    x_{j\to i}^{(t)}
    \right)_{(j,i)\in\mathcal E_{\mathcal B\to\mathcal H}},
\end{equation*}
under any fixed ordering of $\mathcal E_{\mathcal B\to\mathcal H}$.
We then define the stacked communicated state relevant to the honest-node dynamics as
\begin{equation*}
    \Xi^{(t)}
    \coloneqq
    \begin{bmatrix}
        X_{\mathcal H}^{(t)}
        \\
        Z_{\mathcal B}^{(t)}
    \end{bmatrix}.
\end{equation*}
Let
\begin{equation*}
    n_{\Xi}
    \coloneqq
    |\mathcal H|
    +
    m_{\mathcal B\to\mathcal H}.
\end{equation*}
For any time-independent function $f$ of these communicated states, we write
\begin{equation*}
    f^{(t)}
    \coloneqq
    f\left(\Xi^{(t)}\right).
\end{equation*}

\section{Proposed Algorithm}
\label{sec:proposed_algorithm}

In this section, we present the \textit{Adaptive Reputation-based Consensus} (A-RepC) algorithm, which jointly targets fault-tolerant consensus and long-term neighbor reputation learning.

Each node evaluates its neighbors using an outlier-robust discrepancy measure and updates their reputations by balancing historical reliability (exploitation) with diversity for long-term learning (exploration).
The resulting reputations define convex consensus weights that downweight unreliable influences and promote consensus among honest nodes.
This forms a positive feedback loop: better consensus makes Byzantine neighbors easier to identify, while more accurate identification further strengthens consensus.
The full procedure is summarized in Algorithm~\ref{alg:arepc}.

\begin{algorithm}[hbt!]
    \caption{A-RepC}
    \label{alg:arepc}
    \begin{algorithmic}[1]
        \Require Initial state $x_i^{(0)}$, inverse temperature $\eta$, forgetting factor $\lambda$, step-size $\alpha$, total iterations $T$.
        \Ensure Final state $x_i^{(T)}$, reputations $\{p_{ij}^{(T)}\}_{j\in\mathcal{N}_i}$.

        \State Initialize $L_{ij}^{(-1)} \gets 0$ for all $i\in\mathcal{H}$ and $j\in\mathcal{N}_i$

        \For{$t=0,1,\ldots,T-1$}
        \For{each honest node $i$ in parallel}
        \State Exchange states to obtain $\{x_{j\to i}^{(t)}\}_{j\in\mathcal{N}_i}$
        \ForAll{$j\in\mathcal{N}_i$}
        \State $l_{ij}^{(t)} \gets \|x_{j\to i}^{(t)}-\mathrm{cm}_i^{(t)}\|_\infty$
        \State $L_{ij}^{(t)} \gets \lambda L_{ij}^{(t-1)} + l_{ij}^{(t)}$
        \EndFor
        \State $\{p_{ij}^{(t)}\}_{j\in\mathcal{N}_i}
            \gets
            \operatorname{sparsemax}\!\left(\{-\eta L_{ij}^{(t)}\}_{j\in\mathcal{N}_i}\right)$
        \State $x_i^{(t+1)}
            \gets
            (1-\alpha)x_i^{(t)}
            +
            \alpha\sum_{j\in\mathcal{N}_i}p_{ij}^{(t)}x_{j\to i}^{(t)}$
        \EndFor
        \EndFor
    \end{algorithmic}
\end{algorithm}

In Algorithm~\ref{alg:arepc}, the coordinate-wise median of the received neighbor states serves as an outlier-robust reference:
\begin{equation*}
    (\mathrm{cm}_i^{(t)})_k
    =
    \operatorname{median}\!\left(
    \{(x_{j\to i}^{(t)})_k \mid j \in \mathcal{N}_i\}
    \right),
    \quad
    k=1,\dots,d.
\end{equation*}
The sparsemax operator \cite{martins2016softmax} maps discrepancy-based scores to normalized neighbor weights on the probability simplex:
\begin{equation*}
    \operatorname{sparsemax}(z)
    =
    \arg\max_{p \in \Delta^{K-1}}
    \; p^\top z - \frac{1}{2}\|p\|^2,
\end{equation*}
where $\Delta^{K-1}=\{p\in\mathbb{R}^K\mid p\ge0,\sum_{i=1}^K p_i=1\}$.
This encourages weights to align with the scores while avoiding excessive concentration.
For implementation, we next recall a closed-form characterization of sparsemax.

\begin{proposition}[{\cite[Proposition 1]{martins2016softmax}}]
    \label{prop:sparsemax}
    For any $z \in \mathbb{R}^K$, the sparsemax mapping admits the closed-form solution
    \begin{equation*}
        (p)_i = \max\{(z)_i - \tau(z),\, 0\}, \quad i = 1,\dots,K,
    \end{equation*}
    where the threshold $\tau(z)$ is chosen such that
    \begin{equation*}
        \sum_{i=1}^K \max\{(z)_i - \tau(z),\, 0\} = 1.
    \end{equation*}
    Moreover, let $(z)_1 \ge \cdots \ge (z)_K$ denote the sorted entries of $z$, and define
    \begin{equation*}
        k^\star = \max\left\{k \in \{1,\dots,K\} \;\middle|\;
        1 + k\, (z)_k > \sum_{j=1}^k (z)_j \right\}.
    \end{equation*}
    Then the threshold $\tau(z)$ admits the explicit form
    \begin{equation*}
        \tau(z) = \frac{\sum_{j=1}^{k^\star} (z)_j - 1}{k^\star}.
    \end{equation*}
\end{proposition}

\section{Analysis}
\label{sec:analysis}

In this section, we establish the convergence of A-RepC under Byzantine attacks.
The key idea is that adaptive reputation weighting induces a state-dependent truncation mechanism.
Adversarial neighbors with sufficiently large accumulated discrepancies are assigned zero weight, while any retained adversarial influence is constrained by the discounted honest-state dispersion.
Due to the forgetting mechanism, the effective truncation region depends on an exponentially discounted history of the honest-state dispersion.
As the honest states approach consensus, the discounted dispersion also vanishes, and the admissible influence of Byzantine messages correspondingly decreases.
We show that the influence of retained Byzantine states is uniformly controlled throughout the evolution, which leads to the convergence of A-RepC.

For the analysis, define the instantaneous honest-node diameter
\begin{equation*}
    \Delta_{\mathcal{H}}^{(t)}
    \coloneqq
    \max_{i,j \in \mathcal{H}}
    \|x_i^{(t)} - x_j^{(t)}\|_\infty .
\end{equation*}
To characterize the accumulated effect induced by the forgetting factor, we further define the discounted honest-node diameter as
\begin{equation*}
    \widetilde{\Delta}_{\mathcal{H}}^{(t)}
    \coloneqq
    \lambda \widetilde{\Delta}_{\mathcal{H}}^{(t-1)}
    +
    \Delta_{\mathcal{H}}^{(t)},
    \qquad
    \widetilde{\Delta}_{\mathcal{H}}^{(-1)} = 0,
\end{equation*}
or equivalently,
\begin{equation*}
    \widetilde{\Delta}_{\mathcal{H}}^{(t)}
    =
    \sum_{s=0}^{t}
    \lambda^{t-s}
    \Delta_{\mathcal{H}}^{(s)}.
\end{equation*}
For each $i\in\mathcal H$, let
\begin{equation*}
    d_i^{\mathcal H}
    \coloneqq
    |\mathcal N_i\cap\mathcal H|,
\end{equation*}
and define
\begin{equation*}
    d_{\min}^{\mathcal H}
    \coloneqq
    \min_{i\in\mathcal H} d_i^{\mathcal H},
    \qquad
    d_{\max}^{\mathcal H}
    \coloneqq
    \max_{i\in\mathcal H} d_i^{\mathcal H}.
\end{equation*}

\begin{lemma}
    \label{lem:honest_loss_bound}
    Suppose that Assumption~\ref{assump:majority_honest} holds.
    Then, for every honest node $i \in \mathcal{H}$, every honest neighbor $j \in \mathcal{N}_i \cap \mathcal{H}$, and all $t \ge 0$,
    \begin{equation*}
        \|x_j^{(t)} - \mathrm{cm}_i^{(t)}\|_\infty
        \leq
        \Delta_{\mathcal{H}}^{(t)}.
    \end{equation*}
\end{lemma}

\begin{proof}
    By the definition of the coordinate-wise median, for each coordinate
    $\ell=1,\dots,d$,
    \begin{equation*}
        \min_{j \in \mathcal{N}_i \cap \mathcal{H}}
        (x_j^{(t)})_\ell
        \leq
        (\mathrm{cm}_i^{(t)})_\ell
        \leq
        \max_{j \in \mathcal{N}_i \cap \mathcal{H}}
        (x_j^{(t)})_\ell.
    \end{equation*}
    Therefore, we have
    \begin{equation*}
        \|x_j^{(t)} - \mathrm{cm}_i^{(t)}\|_\infty
        \leq
        \max_{u,v \in \mathcal{N}_i \cap \mathcal{H}}
        \|x_u^{(t)} - x_v^{(t)}\|_\infty
        \leq
        \Delta_{\mathcal{H}}^{(t)}.
    \end{equation*}
\end{proof}

As a direct consequence of Lemma~\ref{lem:honest_loss_bound}, the accumulated loss of every honest neighbor $j \in \mathcal{N}_i \cap \mathcal{H}$ satisfies
\begin{equation}
    \label{eq:honest_accumulated_loss_bound}
    L_{ij}^{(t)}
    =
    \sum_{s=0}^{t}
    \lambda^{t-s}
    l_{ij}^{(s)}
    \leq
    \sum_{s=0}^{t}
    \lambda^{t-s}
    \Delta_{\mathcal{H}}^{(s)}
    =
    \widetilde{\Delta}_{\mathcal{H}}^{(t)}.
\end{equation}

\begin{lemma}
    \label{lem:byzantine_remove_bound}
    Suppose that Assumption~\ref{assump:majority_honest} holds and that $\lambda \in [0,1)$.
    For every honest node $i \in \mathcal{H}$ and every neighbor $j \in \mathcal{N}_i$, if $p_{ij}^{(t)} > 0$, then
    \begin{equation*}
        L_{ij}^{(t)}
        <
        \widetilde{\Delta}_{\mathcal{H}}^{(t)}
        +
        \frac{1}{\eta d_i^{\mathcal{H}}}.
    \end{equation*}
\end{lemma}

\begin{proof}
    Fix an honest node $i \in \mathcal{H}$.
    Denote the sparsemax threshold at node $i$ by $\tau_i^{(t)}$.

    When all neighbors' accumulated losses satisfy
    \begin{equation*}
        L_{ij}^{(t)}
        \leq
        \widetilde{\Delta}_{\mathcal{H}}^{(t)},
        \qquad
        \forall j \in \mathcal{N}_i,
    \end{equation*}
    the lemma holds trivially.
    Thus, it remains to consider the case where there exists a neighbor $j \in \mathcal{N}_i$ such that
    \begin{equation*}
        L_{ij}^{(t)}
        >
        \widetilde{\Delta}_{\mathcal{H}}^{(t)}.
    \end{equation*}

    Assume that $p_{ij}^{(t)} > 0$ for some neighbor $j \in \mathcal{N}_i$ satisfying
    \begin{equation*}
        L_{ij}^{(t)}
        >
        \widetilde{\Delta}_{\mathcal{H}}^{(t)}.
    \end{equation*}
    By the definition of sparsemax, we have
    \begin{equation*}
        -\eta L_{ij}^{(t)}
        >
        \tau_i^{(t)}.
    \end{equation*}
    For any honest neighbor $k \in \mathcal{N}_i \cap \mathcal{H}$, it follows from \eqref{eq:honest_accumulated_loss_bound} that
    \begin{equation*}
        -\eta L_{ik}^{(t)}
        \geq
        -\eta \widetilde{\Delta}_{\mathcal{H}}^{(t)}
        >
        -\eta L_{ij}^{(t)}
        >
        \tau_i^{(t)}.
    \end{equation*}
    Hence, $p_{ik}^{(t)} > 0$ for every honest neighbor $k \in \mathcal{N}_i \cap \mathcal{H}$.
    Therefore, by the definition of sparsemax,
    \begin{equation*}
        1
        =
        \sum_{v \in \mathcal{N}_i} p_{iv}^{(t)}
        \geq
        \sum_{k \in \mathcal{N}_i \cap \mathcal{H}} p_{ik}^{(t)}.
    \end{equation*}
    For each honest neighbor $k \in \mathcal{N}_i \cap \mathcal{H}$, we have
    \begin{equation*}
        p_{ik}^{(t)}
        =
        -\eta L_{ik}^{(t)}
        -
        \tau_i^{(t)}
        \geq
        -\eta \widetilde{\Delta}_{\mathcal{H}}^{(t)}
        -
        \tau_i^{(t)}.
    \end{equation*}
    It follows that
    \begin{equation*}
        1
        \geq
        d_i^{\mathcal{H}}
        \left(
        -\eta \widetilde{\Delta}_{\mathcal{H}}^{(t)}
        -
        \tau_i^{(t)}
        \right).
    \end{equation*}
    Since $p_{ij}^{(t)} > 0$, we also have
    \begin{equation*}
        -\tau_i^{(t)}
        >
        \eta L_{ij}^{(t)}.
    \end{equation*}
    Combining the above inequalities gives
    \begin{equation*}
        1
        >
        d_i^{\mathcal{H}}
        \eta
        \left(
        L_{ij}^{(t)}
        -
        \widetilde{\Delta}_{\mathcal{H}}^{(t)}
        \right).
    \end{equation*}
    Therefore,
    \begin{equation*}
        L_{ij}^{(t)}
        <
        \widetilde{\Delta}_{\mathcal{H}}^{(t)}
        +
        \frac{1}{\eta d_i^{\mathcal{H}}}.
    \end{equation*}
\end{proof}

\begin{remark}
    Lemma~\ref{lem:byzantine_remove_bound} shows that the truncation rule adapts to the discounted history of the honest-state dispersion.
    When the honest states remain dispersed, the discounted diameter is large, leading to a relatively loose truncation region and avoiding premature suppression of potentially useful neighbors.
    As the honest states approach consensus, the discounted diameter also converges to zero because $\lambda<1$, and the truncation region correspondingly becomes tighter.
    The term $1 / (\eta d_i^{\mathcal{H}})$ provides an additional margin that prevents small discrepancies from immediately causing truncation near consensus.
\end{remark}

\begin{proposition}
    \label{prop:weights_at_consensus}
    Suppose that Assumption~\ref{assump:majority_honest} holds and that $\lambda\in[0,1)$.
    Let
    \begin{equation*}
        \bar{x}_{\mathcal H}^{(t)}
        \coloneqq
        \frac{1}{|\mathcal H|}
        \sum_{i\in\mathcal H}
        x_i^{(t)}.
    \end{equation*}
    Assume that $\Delta_{\mathcal H}^{(t)}\to0$ and that there exist $T\geq0$ and $\gamma>0$ such that
    \begin{equation*}
        \left\|
        x_{k\to i}^{(t)}
        -
        \bar{x}_{\mathcal H}^{(t)}
        \right\|_\infty
        \geq
        \gamma,
        \qquad
        \forall t\geq T,\
        i\in\mathcal H,\
        k\in\mathcal N_i\cap\mathcal B.
    \end{equation*}
    If
    \begin{equation}
        \label{eq:persistent_separation_condition}
        \gamma
        >
        \frac{1-\lambda}{\eta d_{\min}^{\mathcal H}},
    \end{equation}
    then there exists a finite $T^\star\geq T$ such that
    \begin{equation*}
        p_{ik}^{(t)}
        =
        0,
        \qquad
        \forall t\geq T^\star,\
        i\in\mathcal H,\
        k\in\mathcal N_i\cap\mathcal B.
    \end{equation*}
    Moreover,
    \begin{equation*}
        p_{ij}^{(t)}
        \rightarrow
        \frac{1}{d_i^{\mathcal H}},
        \qquad
        \forall i\in\mathcal H,\
        j\in\mathcal N_i\cap\mathcal H.
    \end{equation*}
\end{proposition}

\begin{proof}
    Since $\Delta_{\mathcal H}^{(t)}\rightarrow0$ and $\lambda\in[0,1)$, the discounted honest-node diameter satisfies
    \begin{equation*}
        \widetilde{\Delta}_{\mathcal H}^{(t)}
        \rightarrow
        0.
    \end{equation*}
    From~\eqref{eq:honest_accumulated_loss_bound}, for every honest neighbor $j\in\mathcal N_i\cap\mathcal H$,
    \begin{equation*}
        0
        \leq
        L_{ij}^{(t)}
        \leq
        \widetilde{\Delta}_{\mathcal H}^{(t)},
    \end{equation*}
    and hence $L_{ij}^{(t)}\rightarrow0$.

    By Assumption~\ref{assump:majority_honest}, the coordinate-wise median lies within the coordinate-wise range of the honest neighbors.
    Therefore,
    \begin{equation*}
        \|\mathrm{cm}_i^{(t)}-\bar{x}_{\mathcal H}^{(t)}\|_\infty
        \leq
        \Delta_{\mathcal H}^{(t)}.
    \end{equation*}
    For every $i\in\mathcal H$, every Byzantine neighbor $k\in\mathcal N_i\cap\mathcal B$, and all $t\geq T$,
    \begin{equation*}
        l_{ik}^{(t)}
        =
        \|x_{k\to i}^{(t)}-\mathrm{cm}_i^{(t)}\|_\infty
        \geq
        \gamma
        -
        \Delta_{\mathcal H}^{(t)}.
    \end{equation*}
    Since $\Delta_{\mathcal H}^{(t)}\rightarrow0$, it follows that
    \begin{equation*}
        \liminf_{t\rightarrow\infty}
        l_{ik}^{(t)}
        \geq
        \gamma.
    \end{equation*}
    Using
    \begin{equation*}
        L_{ik}^{(t)}
        =
        \lambda L_{ik}^{(t-1)}
        +
        l_{ik}^{(t)},
    \end{equation*}
    we obtain
    \begin{equation*}
        \liminf_{t\rightarrow\infty}
        L_{ik}^{(t)}
        \geq
        \frac{\gamma}{1-\lambda}.
    \end{equation*}
    By~\eqref{eq:persistent_separation_condition},
    \begin{equation*}
        \frac{\gamma}{1-\lambda}
        >
        \frac{1}{\eta d_{\min}^{\mathcal H}}
        \geq
        \frac{1}{\eta d_i^{\mathcal H}}.
    \end{equation*}
    Since $\widetilde{\Delta}_{\mathcal H}^{(t)}\rightarrow0$, for each honest--Byzantine neighbor pair $(i,k)$ there exists a finite time after which
    \begin{equation*}
        L_{ik}^{(t)}
        \geq
        \widetilde{\Delta}_{\mathcal H}^{(t)}
        +
        \frac{1}{\eta d_i^{\mathcal H}}.
    \end{equation*}
    Since the network contains only finitely many such pairs, there exists a finite $T^\star\geq T$ such that the above inequality holds for all relevant pairs and all $t\geq T^\star$.
    By the contrapositive of Lemma~\ref{lem:byzantine_remove_bound},
    \begin{equation*}
        p_{ik}^{(t)}
        =
        0,
        \qquad
        \forall t\geq T^\star,\
        i\in\mathcal H,\
        k\in\mathcal N_i\cap\mathcal B.
    \end{equation*}

    After all Byzantine neighbors receive zero weight, the local sparsemax mapping reduces to the honest-neighbor scores.
    Since $L_{ij}^{(t)}\rightarrow0$ for every honest neighbor,
    \begin{equation*}
        \operatorname{sparsemax}
        \left(
        \{-\eta L_{ij}^{(t)}\}_{j\in\mathcal N_i\cap\mathcal H}
        \right)
        \rightarrow
        \operatorname{sparsemax}(0).
    \end{equation*}
    Therefore,
    \begin{equation*}
        p_{ij}^{(t)}
        \rightarrow
        \frac{1}{d_i^{\mathcal H}},
        \qquad
        i\in\mathcal H,\
        j\in\mathcal N_i\cap\mathcal H.
    \end{equation*}
\end{proof}

\begin{remark}
    Proposition~\ref{prop:weights_at_consensus} characterizes the identification behavior of A-RepC in the presence of historical reputation information.
    When the honest states approach consensus, the accumulated losses of honest neighbors vanish because $\lambda<1$.
    In contrast, a Byzantine neighbor that remains persistently separated from the honest-state cluster accumulates a nonvanishing loss.
    The resulting separation condition is
    \begin{equation*}
        \gamma
        >
        \frac{1-\lambda}{\eta d_{\min}^{\mathcal{H}}},
    \end{equation*}
    which reduces to the instantaneous threshold $1/(\eta d_{\min}^{\mathcal{H}})$ when $\lambda=0$.
    Hence, a larger forgetting factor allows persistent Byzantine deviations to be identified even when their instantaneous magnitude is smaller, reflecting the benefit of accumulating evidence over time.
\end{remark}

\begin{lemma}
    \label{lem:loss_lipschitz}
    For each $i\in\mathcal H$, let
    \begin{equation*}
        l_i :
        \mathbb R^{d(|\mathcal H|+|\mathcal E_{\mathcal B\to\mathcal H}|)}
        \to
        \mathbb R^{|\mathcal N_i|}
    \end{equation*}
    denote the instantaneous loss map induced by Algorithm~\ref{alg:arepc}.
    Then $l_i$ is Lipschitz continuous.
    More precisely, for any two communicated-state configurations $\Xi$ and $\Upsilon$,
    \begin{equation*}
        \|l_i(\Xi)-l_i(\Upsilon)\|
        \leq
        2\sqrt{|\mathcal N_i|}
        \|\Xi-\Upsilon\|.
    \end{equation*}
\end{lemma}

\begin{proof}
    Let $x_{j\to i}$ and $y_{j\to i}$ denote the values received by node $i$ from neighbor $j$ under $\Xi$ and $\Upsilon$, respectively.
    By the reverse triangle inequality,
    \begin{equation*}
        \begin{aligned}
             & |l_{ij}(\Xi)-l_{ij}(\Upsilon)| \\
             & \, =
            \left|
            \|x_{j\to i}-\operatorname{cm}_i(\Xi)\|_\infty
            -
            \|y_{j\to i}-\operatorname{cm}_i(\Upsilon)\|_\infty
            \right|                           \\
             & \, \leq
            \|x_{j\to i}-y_{j\to i}\|_\infty
            +
            \|\operatorname{cm}_i(\Xi)-\operatorname{cm}_i(\Upsilon)\|_\infty.
        \end{aligned}
    \end{equation*}

    Since the coordinate-wise median is nonexpansive with respect to the maximum perturbation of its inputs,
    \begin{equation*}
        \|\operatorname{cm}_i(\Xi)-\operatorname{cm}_i(\Upsilon)\|_\infty
        \leq
        \max_{j\in\mathcal N_i}
        \|x_{j\to i}-y_{j\to i}\|_\infty
        \leq
        \|\Xi-\Upsilon\|.
    \end{equation*}
    Moreover, for every $j\in\mathcal N_i$,
    \begin{equation*}
        \|x_{j\to i}-y_{j\to i}\|_\infty
        \leq
        \|\Xi-\Upsilon\|.
    \end{equation*}
    Therefore,
    \begin{equation*}
        |l_{ij}(\Xi)-l_{ij}(\Upsilon)|
        \leq
        2\|\Xi-\Upsilon\|,
        \qquad
        \forall j\in\mathcal N_i.
    \end{equation*}
    It follows that
    \begin{equation*}
        \|l_i(\Xi)-l_i(\Upsilon)\|^2
        \leq
        4|\mathcal N_i|
        \|\Xi-\Upsilon\|^2.
    \end{equation*}
    Taking square roots gives the desired result.
\end{proof}

\begin{lemma}
    \label{lem:weight_lipschitz}
    Consider two communicated-state trajectories
    $\{\Xi^{(s)}\}_{s=0}^{t}$
    and
    $\{\Upsilon^{(s)}\}_{s=0}^{t}$
    with the same initialization
    $L_i^{(-1)}=0$.
    Let $P_{\Xi}^{(t)}$ and $P_{\Upsilon}^{(t)}$ denote the corresponding weight matrices generated by Algorithm~\ref{alg:arepc}.
    Then
    \begin{equation*}
        \|P_{\Xi}^{(t)}-P_{\Upsilon}^{(t)}\|_F
        \leq
        2\eta
        \sqrt{
            \sum_{i\in\mathcal H}
            |\mathcal N_i|
        }
        \sum_{s=0}^{t}
        \lambda^{t-s}
        \|\Xi^{(s)}-\Upsilon^{(s)}\|.
    \end{equation*}
\end{lemma}

\begin{proof}
    For each $i\in\mathcal H$, the accumulated loss vectors satisfy
    \begin{equation*}
        L_{i,\Xi}^{(t)}
        -
        L_{i,\Upsilon}^{(t)}
        =
        \sum_{s=0}^{t}
        \lambda^{t-s}
        \left(
        l_i(\Xi^{(s)})
        -
        l_i(\Upsilon^{(s)})
        \right).
    \end{equation*}
    By Lemma~\ref{lem:loss_lipschitz},
    \begin{equation*}
        \left\|
        L_{i,\Xi}^{(t)}
        -
        L_{i,\Upsilon}^{(t)}
        \right\|
        \leq
        2\sqrt{|\mathcal N_i|}
        \sum_{s=0}^{t}
        \lambda^{t-s}
        \|\Xi^{(s)}-\Upsilon^{(s)}\|.
    \end{equation*}
    Since sparsemax is nonexpansive,
    \begin{equation*}
        \left\|
        P_{i,\Xi}^{(t)}
        -
        P_{i,\Upsilon}^{(t)}
        \right\|
        \leq
        \eta
        \left\|
        L_{i,\Xi}^{(t)}
        -
        L_{i,\Upsilon}^{(t)}
        \right\|.
    \end{equation*}
    Therefore,
    \begin{equation*}
        \left\|
        P_{i,\Xi}^{(t)}
        -
        P_{i,\Upsilon}^{(t)}
        \right\|
        \leq
        2\eta
        \sqrt{|\mathcal N_i|}
        \sum_{s=0}^{t}
        \lambda^{t-s}
        \|\Xi^{(s)}-\Upsilon^{(s)}\|.
    \end{equation*}
    Stacking the rows corresponding to all $i\in\mathcal H$ gives
    \begin{equation*}
        \|P_{\Xi}^{(t)}-P_{\Upsilon}^{(t)}\|_F
        \leq
        2\eta
        \sqrt{
            \sum_{i\in\mathcal H}
            |\mathcal N_i|
        }
        \sum_{s=0}^{t}
        \lambda^{t-s}
        \|\Xi^{(s)}-\Upsilon^{(s)}\|.
    \end{equation*}
\end{proof}

We now turn to the general convergence analysis.
Lemma~\ref{lem:byzantine_remove_bound} shows that neighbors whose accumulated losses are sufficiently large relative to the discounted honest-state dispersion are removed by the truncation rule.
However, this pointwise truncation property does not provide a uniform-in-time bound on Byzantine messages that remain admissible, particularly when Byzantine behaviors vary over time.
We therefore impose the following boundedness assumption.

\begin{assumption}[Bounded Byzantine messages]
    \label{assump:bounded_byzantine}
    There exists a constant $M_{\mathcal B}>0$ such that
    \begin{equation*}
        \|x_{j\to i}^{(t)}\|
        \leq
        M_{\mathcal B},
        \qquad
        \forall i\in\mathcal H,\
        j\in\mathcal N_i\cap\mathcal B,\
        t\geq0.
    \end{equation*}
\end{assumption}

Under this assumption, Byzantine neighbors may still be assigned positive weights, but the magnitudes of their transmitted values remain uniformly bounded.

We next introduce the graph-theoretic notation used in the convergence analysis.
Let $A_{\mathcal H}$ and $D_{\mathcal H}$ denote the adjacency matrix and degree matrix of the honest subgraph, respectively.
Accordingly,
\begin{equation*}
    D_{\mathcal H}
    =
    \operatorname{diag}
    \left(
    \{d_i^{\mathcal H}\}_{i\in\mathcal H}
    \right).
\end{equation*}
Define the nominal honest-node mixing matrix as
\begin{equation*}
    P_{\mathcal H}^{\star}
    \coloneqq
    D_{\mathcal H}^{-1}A_{\mathcal H}.
\end{equation*}
The associated symmetric normalized Laplacian is
\begin{equation*}
    L_{\mathrm{sym},\mathcal H}
    \coloneqq
    I
    -
    D_{\mathcal H}^{-1/2}
    A_{\mathcal H}
    D_{\mathcal H}^{-1/2}.
\end{equation*}
Let $\pi\in\mathbb R^{|\mathcal H|}$ denote the stationary distribution associated with $P_{\mathcal H}^{\star}$, satisfying
\begin{equation*}
    P_{\mathcal H}^{\star}\mathbf 1
    =
    \mathbf 1,
    \qquad
    \pi^\top P_{\mathcal H}^{\star}
    =
    \pi^\top,
    \qquad
    \pi^\top\mathbf 1
    =
    1.
\end{equation*}
For the undirected honest subgraph considered here, the stationary distribution is
\begin{equation*}
    \pi_i
    =
    \frac{
    d_i^{\mathcal H}
    }{
    \sum_{\ell\in\mathcal H}d_\ell^{\mathcal H}
    },
    \qquad
    i\in\mathcal H.
\end{equation*}
Define the stationary-distribution disagreement projector and the corresponding disagreement vector as
\begin{equation*}
    \Pi_\pi
    \coloneqq
    \left(
    I_{|\mathcal H|}
    -
    \mathbf 1\pi^\top
    \right)
    \otimes I_d,
    \qquad
    E_{\mathcal H}^{(t)}
    \coloneqq
    \Pi_\pi X_{\mathcal H}^{(t)}.
\end{equation*}
Let $P_{\mathcal H}^{(t)}$ denote the honest-neighbor weight matrix.
We further define the Byzantine-message weight matrix
$P_{\mathcal B\to\mathcal H}^{(t)}
    \in
    \mathbb R^{|\mathcal H|\times m_{\mathcal B\to\mathcal H}}$,
whose columns are indexed by $(k,r)\in\mathcal E_{\mathcal B\to\mathcal H}$, with
\begin{equation*}
    \left[
        P_{\mathcal B\to\mathcal H}^{(t)}
        \right]_{i,(k,r)}
    =
    \begin{cases}
        p_{ik}^{(t)}, & r=i,     \\
        0,            & r\neq i.
    \end{cases}
\end{equation*}
Define the expanded adaptive mixing matrix as
\begin{equation*}
    \widetilde P^{(t)}
    \coloneqq
    \begin{bmatrix}
        P_{\mathcal H}^{(t)}
         &
        P_{\mathcal B\to\mathcal H}^{(t)}
    \end{bmatrix}.
\end{equation*}
We equip the honest-node state space with the degree-weighted norm
\begin{equation*}
    \|Z\|_{D_{\mathcal H}}
    \coloneqq
    \left\|
    \left(
    D_{\mathcal H}^{1/2}
    \otimes I_d
    \right)
    Z
    \right\|_2.
\end{equation*}
For notational simplicity, whenever a scalar-weight matrix acts on stacked $d$-dimensional vectors, we use the same symbol to denote its Kronecker lifting with $I_d$.
To state the main result compactly, for a given $\alpha\in(0,1)$, define
\begin{equation*}
    \rho_0
    \coloneqq
    \max_{\ell=2,\ldots,|\mathcal H|}
    |1-\alpha\mu_\ell|,
\end{equation*}
where $0=\mu_1<\mu_2\leq\cdots\leq\mu_{|\mathcal H|}\leq2$ are the eigenvalues of $L_{\mathrm{sym},\mathcal H}$.
Let
\begin{equation*}
    M
    \coloneqq
    \max\{M_{\mathcal H},M_{\mathcal B}\},
    \qquad
    M_{\mathcal H}
    \coloneqq
    \max_{i\in\mathcal H}\|x_i^{(0)}\|.
\end{equation*}
Define
\begin{equation*}
    C_1
    \coloneqq
    2M
    \sqrt{
    \frac{
    n_{\Xi}d_{\max}^{\mathcal H}
    }{
    d_{\min}^{\mathcal H}
    }
    \sum_{i\in\mathcal H}|\mathcal N_i|
    },
\end{equation*}
and
\begin{equation*}
    C_2
    \coloneqq
    \frac{
    2\sqrt{
    n_{\Xi}d_{\max}^{\mathcal H}
    m_{\mathcal B\to\mathcal H}
    \sum_{i\in\mathcal H}|\mathcal N_i|
    }
    }{
    d_{\min}^{\mathcal H}
    }.
\end{equation*}

\begin{theorem}
    \label{thm:bounded_byzantine_ISS}
    Suppose that Assumptions~\ref{assump:majority_honest},~\ref{assump:honest_connected}, and~\ref{assump:bounded_byzantine} hold, and that $\lambda\in[0,1)$.
    For any fixed $\alpha\in(0,1)$, if $\eta>0$ satisfies
    \begin{equation}
        \label{eq:augmented_stability_condition}
        \alpha\eta C_1
        <
        (1-\rho_0)(1-\lambda),
    \end{equation}
    then there exist constants $\kappa\geq1$ and $\rho\in(0,1)$ such that, for all $t\geq0$,
    \begin{equation*}
        \left\|E_{\mathcal H}^{(t)}\right\|_{D_{\mathcal H}}
        \leq
        \kappa\rho^t\left\|E_{\mathcal H}^{(0)}\right\|_{D_{\mathcal H}}
        +
        \frac{\alpha C_2 M}{(1-\rho_0)(1-\lambda)-\alpha\eta C_1}.
    \end{equation*}
    Consequently, the honest nodes approach a bounded neighborhood of the consensus manifold.
\end{theorem}

\begin{proof}
    By Assumption~\ref{assump:bounded_byzantine} and the convex-combination update of the honest nodes, induction gives
    \begin{equation}
        \label{eq:global_state_bound}
        \|\Xi^{(t)}\|
        \leq
        \sqrt{n_{\Xi}}M,
        \qquad
        \forall t\geq0.
    \end{equation}

    Define the ideal expanded mixing matrix as
    \begin{equation*}
        \widetilde P^{\star}
        \coloneqq
        \begin{bmatrix}
            P_{\mathcal H}^{\star}
             &
            0
        \end{bmatrix},
    \end{equation*}
    and let
    \begin{equation*}
        W_{\mathcal H}^{\star}
        \coloneqq
        (1-\alpha)I
        +
        \alpha P_{\mathcal H}^{\star}.
    \end{equation*}
    The honest-node dynamics can then be written as
    \begin{equation*}
        X_{\mathcal H}^{(t+1)}
        =
        W_{\mathcal H}^{\star}
        X_{\mathcal H}^{(t)}
        +
        \alpha
        \left(
        \widetilde P^{(t)}
        -
        \widetilde P^{\star}
        \right)
        \Xi^{(t)}.
    \end{equation*}
    Since
    \begin{equation*}
        W_{\mathcal H}^{\star}\mathbf 1
        =
        \mathbf 1,
        \qquad
        \pi^\top W_{\mathcal H}^{\star}
        =
        \pi^\top,
    \end{equation*}
    we have
    \begin{equation*}
        \Pi_\pi W_{\mathcal H}^{\star}
        =
        W_{\mathcal H}^{\star}\Pi_\pi.
    \end{equation*}
    Hence,
    \begin{equation}
        \label{eq:weighted_disagreement_dynamics}
        E_{\mathcal H}^{(t+1)}
        =
        W_{\mathcal H}^{\star}
        E_{\mathcal H}^{(t)}
        +
        \alpha
        \Pi_\pi
        \left(
        \widetilde P^{(t)}
        -
        \widetilde P^{\star}
        \right)
        \Xi^{(t)}.
    \end{equation}

    Moreover,
    \begin{equation*}
        D_{\mathcal H}^{1/2}
        W_{\mathcal H}^{\star}
        D_{\mathcal H}^{-1/2}
        =
        I
        -
        \alpha
        L_{\mathrm{sym},\mathcal H}.
    \end{equation*}
    By the definition of $\rho_0$ and Assumption~\ref{assump:honest_connected}, we have $\rho_0<1$.
    Therefore,
    \begin{equation}
        \label{eq:nominal_weighted_contraction}
        \left\|
        W_{\mathcal H}^{\star}
        E_{\mathcal H}^{(t)}
        \right\|_{D_{\mathcal H}}
        \leq
        \rho_0
        \left\|
        E_{\mathcal H}^{(t)}
        \right\|_{D_{\mathcal H}}.
    \end{equation}

    We next bound the deviation of the adaptive weight matrix from the ideal weight matrix.
    Let
    \begin{equation*}
        \bar{x}_{\pi,\mathcal H}^{(s)}
        \coloneqq
        \sum_{i\in\mathcal H}
        \pi_i x_i^{(s)}.
    \end{equation*}
    For each $s\geq0$, construct a reference communicated state $\widehat\Xi^{(s)}$ by setting
    \begin{equation*}
        \widehat x_i^{(s)}
        =
        \bar{x}_{\pi,\mathcal H}^{(s)},
        \qquad
        \forall i\in\mathcal H.
    \end{equation*}
    For each $(k,i)\in\mathcal E_{\mathcal B\to\mathcal H}$, if
    \begin{equation*}
        \left\|
        x_{k\to i}^{(s)}
        -
        \bar{x}_{\pi,\mathcal H}^{(s)}
        \right\|_\infty
        \geq
        \frac{1}{\eta d_{\min}^{\mathcal H}},
    \end{equation*}
    set $\widehat x_{k\to i}^{(s)}=x_{k\to i}^{(s)}$.
    Otherwise, move one coordinate of $x_{k\to i}^{(s)}$ away from $\bar{x}_{\pi,\mathcal H}^{(s)}$ until
    \begin{equation*}
        \left\|
        \widehat x_{k\to i}^{(s)}
        -
        \bar{x}_{\pi,\mathcal H}^{(s)}
        \right\|_\infty
        =
        \frac{1}{\eta d_{\min}^{\mathcal H}}.
    \end{equation*}
    Since each Byzantine-message block is modified by at most $1/(\eta d_{\min}^{\mathcal H})$, this construction guarantees
    \begin{equation}
        \label{eq:reference_state_distance}
        \left\|
        \Xi^{(s)}
        -
        \widehat\Xi^{(s)}
        \right\|
        \leq
        \left\|
        E_{\mathcal H}^{(s)}
        \right\|
        +
        \frac{
            \sqrt{m_{\mathcal B\to\mathcal H}}
        }{
            \eta d_{\min}^{\mathcal H}
        }.
    \end{equation}

    Since all honest reference states coincide and each honest node has a strict majority of honest neighbors, the coordinate-wise median of the reference neighbor messages satisfies
    \begin{equation*}
        \widehat{\mathrm{cm}}_i^{(s)}
        =
        \bar{x}_{\pi,\mathcal H}^{(s)}.
    \end{equation*}
    Hence, for every honest neighbor $j\in\mathcal N_i\cap\mathcal H$,
    \begin{equation*}
        \widehat l_{ij}^{(s)}
        =
        0,
    \end{equation*}
    while for every Byzantine neighbor $k\in\mathcal N_i\cap\mathcal B$,
    \begin{equation*}
        \widehat l_{ik}^{(s)}
        \geq
        \frac{1}{\eta d_{\min}^{\mathcal H}}.
    \end{equation*}
    Therefore,
    \begin{equation*}
        \widehat L_{ij}^{(t)}
        =
        0,
        \qquad
        j\in\mathcal N_i\cap\mathcal H,
    \end{equation*}
    and
    \begin{equation*}
        \widehat L_{ik}^{(t)}
        \geq
        \frac{1}{\eta d_{\min}^{\mathcal H}}
        \sum_{s=0}^{t}\lambda^{t-s}
        \geq
        \frac{1}{\eta d_i^{\mathcal H}},
        \qquad
        k\in\mathcal N_i\cap\mathcal B.
    \end{equation*}
    It follows from the definition of sparsemax that the expanded reference weight matrix is exactly the ideal matrix,
    \begin{equation*}
        \widehat{\widetilde P}^{(t)}
        =
        \widetilde P^{\star},
        \qquad
        \forall t\geq0.
    \end{equation*}

    Applying Lemma~\ref{lem:weight_lipschitz} to the actual and reference communicated-state histories gives
    \begin{equation*}
        \left\|
        \widetilde P^{(t)}
        -
        \widetilde P^{\star}
        \right\|_F
        \leq
        2\eta
        \sqrt{
            \sum_{i\in\mathcal H}
            |\mathcal N_i|
        }
        \sum_{s=0}^{t}
        \lambda^{t-s}
        \left\|
        \Xi^{(s)}
        -
        \widehat\Xi^{(s)}
        \right\|.
    \end{equation*}
    Substituting~\eqref{eq:reference_state_distance} gives
    \begin{equation*}
        \begin{aligned}
            \left\|
            \widetilde P^{(t)}
            -
            \widetilde P^{\star}
            \right\|_F
            \leq{} &
            2\eta
            \sqrt{
                \sum_{i\in\mathcal H}
                |\mathcal N_i|
            }
            \sum_{s=0}^{t}
            \lambda^{t-s}
            \left\|
            E_{\mathcal H}^{(s)}
            \right\|
            \\
                   & +
            \frac{
                2
                \sqrt{
                    m_{\mathcal B\to\mathcal H}
                    \sum_{i\in\mathcal H}
                    |\mathcal N_i|
                }
            }{
                d_{\min}^{\mathcal H}
            }
            \sum_{s=0}^{t}
            \lambda^{t-s}.
        \end{aligned}
    \end{equation*}
    Since
    \begin{equation*}
        \left\|
        E_{\mathcal H}^{(s)}
        \right\|
        \leq
        \frac{1}{
            \sqrt{d_{\min}^{\mathcal H}}
        }
        \left\|
        E_{\mathcal H}^{(s)}
        \right\|_{D_{\mathcal H}},
    \end{equation*}
    and
    \begin{equation*}
        \sum_{s=0}^{t}
        \lambda^{t-s}
        \leq
        \frac{1}{1-\lambda},
    \end{equation*}
    we obtain
    \begin{equation}
        \label{eq:P_difference_weighted_memory}
        \begin{aligned}
            \left\|
            \widetilde P^{(t)}
            -
            \widetilde P^{\star}
            \right\|_F
            \leq{} &
            \frac{
                2\eta
                \sqrt{
                    \sum_{i\in\mathcal H}
                    |\mathcal N_i|
                }
            }{
                \sqrt{d_{\min}^{\mathcal H}}
            }
            \sum_{s=0}^{t}
            \lambda^{t-s}
            \left\|
            E_{\mathcal H}^{(s)}
            \right\|_{D_{\mathcal H}}
            \\
                   & +
            \frac{
                2
                \sqrt{
                    m_{\mathcal B\to\mathcal H}
                    \sum_{i\in\mathcal H}
                    |\mathcal N_i|
                }
            }{
                d_{\min}^{\mathcal H}(1-\lambda)
            }.
        \end{aligned}
    \end{equation}

    Define
    \begin{equation*}
        e_t
        \coloneqq
        \left\|
        E_{\mathcal H}^{(t)}
        \right\|_{D_{\mathcal H}},
        \qquad
        q_t
        \coloneqq
        \sum_{s=0}^{t}
        \lambda^{t-s}e_s.
    \end{equation*}
    Since $\Pi_\pi$ is the orthogonal projector onto the disagreement subspace under the $D_{\mathcal H}$-weighted inner product, it is nonexpansive:
    \begin{equation*}
        \|\Pi_\pi Z\|_{D_{\mathcal H}}
        \leq
        \|Z\|_{D_{\mathcal H}}.
    \end{equation*}
    Moreover,
    \begin{equation*}
        \|Z\|_{D_{\mathcal H}}
        \leq
        \sqrt{d_{\max}^{\mathcal H}}
        \|Z\|.
    \end{equation*}
    Combining these inequalities with~\eqref{eq:global_state_bound} and~\eqref{eq:P_difference_weighted_memory} yields
    \begin{equation}
        \label{eq:weighted_perturbation_bound_memory}
        \alpha
        \left\|
        \Pi_\pi
        \left(
        \widetilde P^{(t)}
        -
        \widetilde P^{\star}
        \right)
        \Xi^{(t)}
        \right\|_{D_{\mathcal H}}
        \leq
        \alpha\eta C_1 q_t
        +
        \frac{\alpha C_2}{1-\lambda}M.
    \end{equation}

    Combining~\eqref{eq:weighted_disagreement_dynamics},~\eqref{eq:nominal_weighted_contraction}, and~\eqref{eq:weighted_perturbation_bound_memory} gives
    \begin{equation*}
        e_{t+1}
        \leq
        \rho_0e_t
        +
        \alpha\eta C_1q_t
        +
        \frac{\alpha C_2}{1-\lambda}M.
    \end{equation*}
    By the definition of $q_t$,
    \begin{equation*}
        q_{t+1}
        =
        e_{t+1}
        +
        \lambda q_t.
    \end{equation*}
    Therefore,
    \begin{equation}
        \label{eq:augmented_comparison_system}
        \begin{bmatrix}
            e_{t+1}
            \\
            q_{t+1}
        \end{bmatrix}
        \preceq
        A_{\lambda,\eta}
        \begin{bmatrix}
            e_t
            \\
            q_t
        \end{bmatrix}
        +
        \frac{\alpha C_2}{1-\lambda}
        M
        \begin{bmatrix}
            1
            \\
            1
        \end{bmatrix},
    \end{equation}
    where
    \begin{equation*}
        A_{\lambda,\eta}
        =
        \begin{bmatrix}
            \rho_0
             &
            \alpha\eta C_1
            \\
            \rho_0
             &
            \lambda+\alpha\eta C_1
        \end{bmatrix}.
    \end{equation*}

    The characteristic polynomial of $A_{\lambda,\eta}$ is
    \begin{equation*}
        \chi(z)
        =
        z^2
        -
        \left(
        \rho_0+\lambda+\alpha\eta C_1
        \right)z
        +
        \rho_0\lambda.
    \end{equation*}
    Since $\rho_0\in[0,1)$ and $\lambda\in[0,1)$, condition~\eqref{eq:augmented_stability_condition} implies
    \begin{equation*}
        1
        -
        \operatorname{tr}(A_{\lambda,\eta})
        +
        \det(A_{\lambda,\eta})
        =
        (1-\rho_0)(1-\lambda)
        -
        \alpha\eta C_1
        >
        0,
    \end{equation*}
    together with
    \begin{equation*}
        1-\det(A_{\lambda,\eta})
        =
        1-\rho_0\lambda
        >
        0.
    \end{equation*}
    The remaining second-order Schur condition is automatically satisfied because
    \begin{equation*}
        1
        +
        \operatorname{tr}(A_{\lambda,\eta})
        +
        \det(A_{\lambda,\eta})
        >
        0.
    \end{equation*}
    Hence,
    \begin{equation*}
        \rho(A_{\lambda,\eta})
        <
        1.
    \end{equation*}

    Since $A_{\lambda,\eta}$ is nonnegative and Schur stable, iterating~\eqref{eq:augmented_comparison_system} gives
    \begin{equation*}
        \begin{bmatrix}
            e_t
            \\
            q_t
        \end{bmatrix}
        \preceq
        A_{\lambda,\eta}^t
        \begin{bmatrix}
            e_0
            \\
            e_0
        \end{bmatrix}
        +
        \frac{\alpha C_2}{1-\lambda}
        M
        \sum_{s=0}^{t-1}
        A_{\lambda,\eta}^s
        \begin{bmatrix}
            1
            \\
            1
        \end{bmatrix}.
    \end{equation*}
    Because $\rho(A_{\lambda,\eta})<1$, there exist constants $\kappa\geq1$ and $\rho\in(0,1)$ such that
    \begin{equation*}
        \|A_{\lambda,\eta}^t\|
        \leq
        \kappa\rho^t.
    \end{equation*}
    Moreover,
    \begin{equation*}
        \det
        \left(
        I-A_{\lambda,\eta}
        \right)
        =
        (1-\rho_0)(1-\lambda)
        -
        \alpha\eta C_1.
    \end{equation*}
    A direct calculation gives
    \begin{equation*}
        \begin{bmatrix}
            1 & 0
        \end{bmatrix}
        \left(
        I-A_{\lambda,\eta}
        \right)^{-1}
        \begin{bmatrix}
            1
            \\
            1
        \end{bmatrix}
        =
        \frac{
            1-\lambda
        }{
            (1-\rho_0)(1-\lambda)
            -
            \alpha\eta C_1
        }.
    \end{equation*}
    Therefore,
    \begin{equation*}
        e_t
        \leq
        \kappa\rho^t e_0
        +
        \frac{
            \alpha C_2
        }{
            (1-\rho_0)(1-\lambda)
            -
            \alpha\eta C_1
        }
        M.
    \end{equation*}
    Thus, the honest-node disagreement converges to a bounded neighborhood of the consensus manifold.
\end{proof}

\begin{corollary}
    \label{cor:separated_byzantine_consensus}
    Suppose that Assumptions~\ref{assump:majority_honest},~\ref{assump:honest_connected}, and~\ref{assump:bounded_byzantine} hold, and that $\lambda\in[0,1)$.
    Assume that
    \begin{equation*}
        \left\|
        x_{j\to i}^{(t)}
        -
        \bar{x}_{\pi,\mathcal H}^{(t)}
        \right\|_\infty
        \geq
        \frac{1}{\eta d_{\min}^{\mathcal H}},
        \qquad
        \forall t\geq0,\
        i\in\mathcal H,\
        j\in\mathcal N_i\cap\mathcal B,
    \end{equation*}
    where
    \begin{equation*}
        \bar{x}_{\pi,\mathcal H}^{(t)}
        \coloneqq
        \sum_{i\in\mathcal H}
        \pi_i x_i^{(t)}.
    \end{equation*}
    If $\alpha\in(0,1)$ and $\eta>0$ satisfy~\eqref{eq:augmented_stability_condition}, then there exist constants $\kappa\geq1$ and $\rho\in(0,1)$ such that, for all $t\geq0$,
    \begin{equation*}
        \left\|
        E_{\mathcal H}^{(t)}
        \right\|_{D_{\mathcal H}}
        \leq
        \kappa\rho^t
        \left\|
        E_{\mathcal H}^{(0)}
        \right\|_{D_{\mathcal H}}.
    \end{equation*}
    Consequently, the honest nodes asymptotically reach consensus with linear convergence.
\end{corollary}

\begin{proof}
    Under the separation condition, the Byzantine reference messages constructed in the proof of Theorem~\ref{thm:bounded_byzantine_ISS} can be chosen as
    \begin{equation*}
        \widehat{x}_{j\to i}^{(s)}
        =
        x_{j\to i}^{(s)},
        \qquad
        (j,i)\in\mathcal E_{\mathcal B\to\mathcal H},
    \end{equation*}
    for every $s\geq0$.
    Hence,
    \begin{equation*}
        \left\|
        \Xi^{(s)}
        -
        \widehat\Xi^{(s)}
        \right\|
        =
        \left\|
        E_{\mathcal H}^{(s)}
        \right\|,
    \end{equation*}
    and the additive Byzantine residual in the perturbation bound of Theorem~\ref{thm:bounded_byzantine_ISS} vanishes.

    Following the same augmented-state argument as in the proof of Theorem~\ref{thm:bounded_byzantine_ISS}, the honest-node disagreement is therefore governed by the corresponding homogeneous comparison system.
    Condition~\eqref{eq:augmented_stability_condition} ensures that its comparison matrix $A_{\lambda,\eta}$ is Schur stable.
    Hence, there exist constants $\kappa\geq1$ and $\rho\in(0,1)$ such that
    \begin{equation*}
        \left\|
        E_{\mathcal H}^{(t)}
        \right\|_{D_{\mathcal H}}
        \leq
        \kappa\rho^t
        \left\|
        E_{\mathcal H}^{(0)}
        \right\|_{D_{\mathcal H}},
        \qquad
        \forall t\geq0.
    \end{equation*}
    This proves linear convergence of the honest-node disagreement to zero.
\end{proof}

\begin{remark}
    Proposition~\ref{prop:weights_at_consensus} and Corollary~\ref{cor:separated_byzantine_consensus} characterize two complementary effects of the forgetting mechanism.
    Proposition~\ref{prop:weights_at_consensus} shows that a persistent Byzantine separation satisfying
    \begin{equation*}
        \gamma
        >
        \frac{1-\lambda}{\eta d_{\min}^{\mathcal H}}
    \end{equation*}
    is sufficient for eventual Byzantine rejection as the honest states approach consensus.
    Corollary~\ref{cor:separated_byzantine_consensus} imposes the stronger pointwise condition
    \begin{equation*}
        \left\|
        x_{j\to i}^{(t)}
        -
        \bar{x}_{\pi,\mathcal H}^{(t)}
        \right\|_\infty
        \geq
        \frac{1}{\eta d_{\min}^{\mathcal H}},
        \qquad
        \forall i\in\mathcal H,\
        j\in\mathcal N_i\cap\mathcal B,
    \end{equation*}
    throughout the evolution, which eliminates the Byzantine residual in the transient analysis and yields linear convergence to exact consensus.
    When $\lambda=0$, the two separation thresholds coincide.
    For $\lambda>0$, historical accumulation allows smaller persistent Byzantine deviations to be identified asymptotically, while the stability condition~\eqref{eq:augmented_stability_condition} becomes more restrictive as $\lambda$ approaches one.
\end{remark}

\section{Performance Evaluation}
\label{sec:experiments}

This section evaluates the performance of the proposed algorithm through two complementary experiments.
The first experiment provides a comprehensive comparison across all considered methods in a setting applicable to every baseline.
The second experiment examines performance in a more challenging high-dimensional scenario with mixed Byzantine attacks, where only the applicable baselines are compared.\footnote{The source code for all examples is publicly available at \url{https://github.com/rui-huang-opt/arepc}.}

Across both settings, consensus accuracy is quantified using Root Mean Square Error (RMSE) and Drift from the Initial Average (DIA), defined as
\begin{equation*}
    \begin{aligned}
         & \mathrm{RMSE}^{(t)} = \sqrt{\frac{1}{|\mathcal{H}|} \sum_{i \in \mathcal{H}} \| x_i^{(t)} - \bar{x}_{\mathcal{H}}^{(t)} \|^2}, \\
         & \mathrm{DIA}^{(t)} = \| \bar{x}_{\mathcal{H}}^{(t)} - \bar{x}_{\mathcal{H}}^{(0)} \|.
    \end{aligned}
\end{equation*}

To facilitate the comparison of reputation dynamics across methods, we express the baseline methods under a unified iterative form:
\begin{equation*}
    x_i^{(t+1)} = (1-\alpha)x_i^{(t)} + \alpha \hat{x}_i^{(t)},
\end{equation*}
where $\hat{x}_i^{(t)}$ denotes the method-dependent update target for node $i$, constructed from its neighbors' information, and $\alpha$ is the step size. This form is exactly equivalent to the original update for some methods, while for others it is obtained through a minor reformulation that preserves the essential update mechanism.

\subsection{Comprehensive Comparison with Baselines}

We benchmark A-RepC against W-MSR~\cite{leblanc2013resilient}, ADRC~\cite{abbas2022resilient}, RepC~\cite{ramos2023discrete}, and WLA~\cite{hou2023resilient}.
W-MSR is extended to vector states coordinate-wise and is included only as a heuristic baseline, since no theoretical guarantee is available in this case. Although RepC was originally proposed for scalar states, it extends naturally to vector states because inter-state discrepancies are mapped to scalar trust scores.

We evaluate all methods in a multi-process environment using TCP-based communication implemented with ZeroMQ and Pyre.
The network consists of $50$ honest nodes and $10$ Byzantine nodes (Fig.~\ref{fig:large_scale_network_topology}), with node states in $\mathbb{R}^4$.
The initial states of the honest nodes are drawn uniformly from $[-100,100]^4$, while each Byzantine node is initialized in the same way and then keeps that value fixed throughout the experiment.
Unless otherwise specified, the baseline parameters follow their original formulations: $f=1$ for W-MSR and RepC, $\epsilon=0.001$ for RepC, and $\theta=0.001$ for WLA.
For A-RepC, we set $\eta=0.001$ and $\lambda=0.8$.

\begin{figure}[hbt!]
    \centering
    \includegraphics[width=0.45\textwidth]{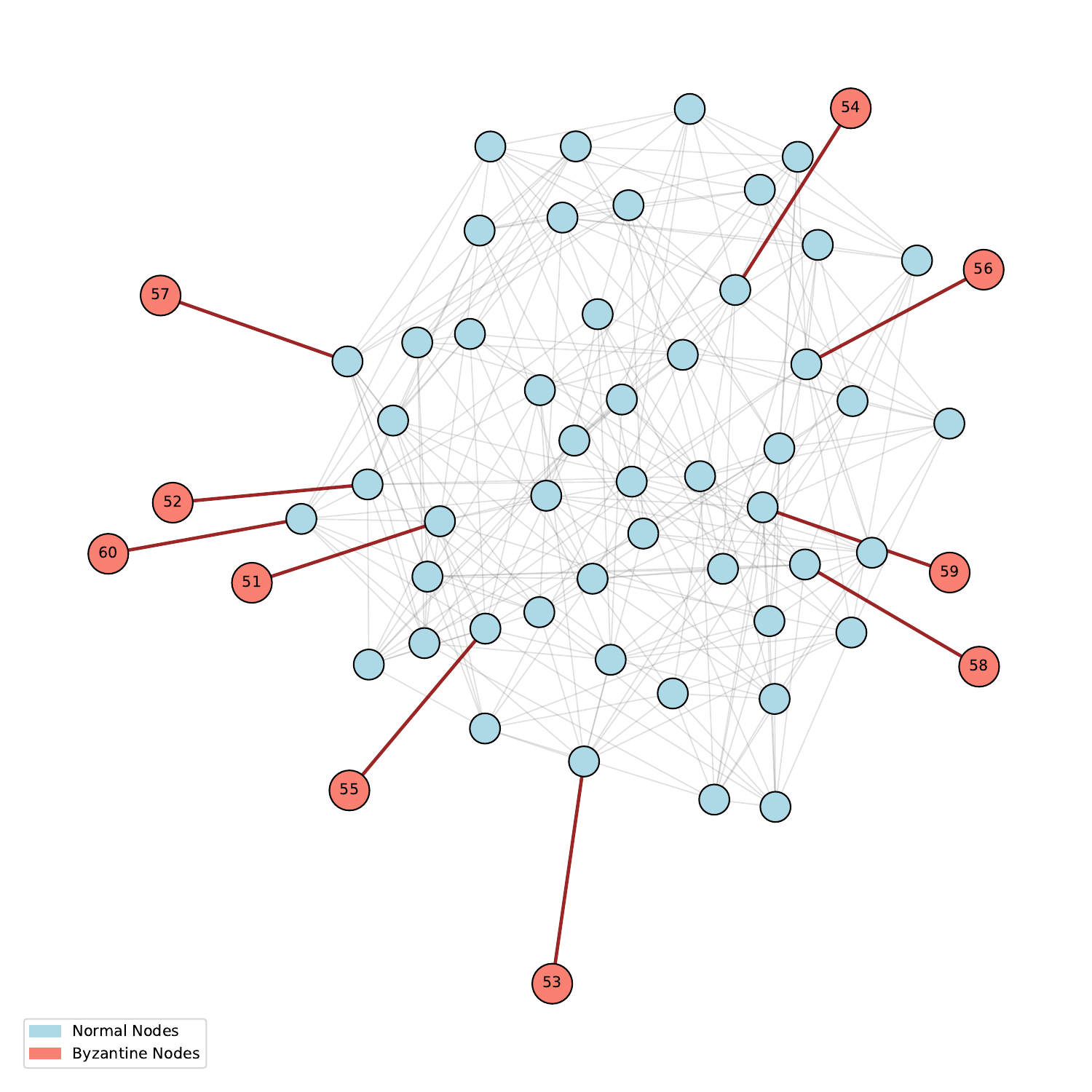}
    \caption{Illustration of the benchmark network topology.}
    \label{fig:large_scale_network_topology}
\end{figure}

Figure~\ref{fig:rmse_and_dia_large_scale} shows that all methods converge at comparable rates and remain close to the initial average.
These results indicate that A-RepC achieves consensus accuracy comparable to the baselines in this large-scale, low-dimensional setting.

\begin{figure}[hbt!]
    \centering
    \includegraphics[width=1.0\linewidth]{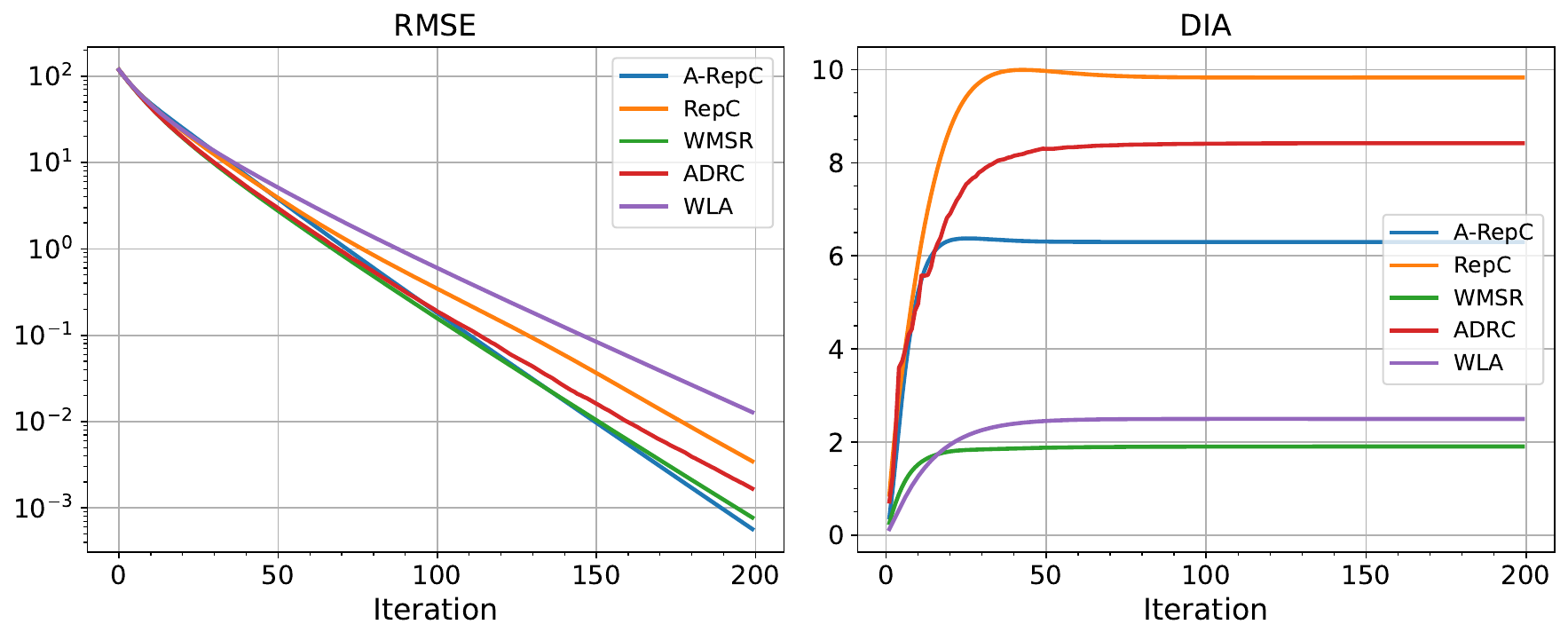}
    \caption{RMSE and DIA evolutions of different methods under fixed-value Byzantine attacks in the large-scale low-dimensional benchmark.}
    \label{fig:rmse_and_dia_large_scale}
\end{figure}

Although the methods exhibit similar overall convergence behavior, their reputation dynamics differ markedly (Fig.~\ref{fig:reputation_evolutions_large_scale}).
A-RepC explicitly identifies adversarial behavior by driving Byzantine reputations to zero while maintaining nearly uniform weights over honest neighbors.
RepC also shows some identification capability, but its reputation evolution is strongly affected by the preset parameter $f$. Consequently, RepC performs reasonably when the true number of adversaries matches $f$, but it may still distort reputations unnecessarily even when no Byzantine neighbors are present, revealing its reliance on prior knowledge.
WLA can also suppress Byzantine neighbors, but unlike A-RepC it does not assign nearly uniform weights to honest neighbors.
Instead, its weights over honest neighbors remain uneven because WLA retains the full history of accumulated losses, so discrepancies arising from initial state differences continue to influence the weight evolution.
In contrast, W-MSR and ADRC rely solely on robust aggregation and do not explicitly identify Byzantine nodes.

\begin{figure}[hbt!]
    \centering

    \begin{minipage}{\linewidth}
        \centering
        \includegraphics[width=\linewidth]{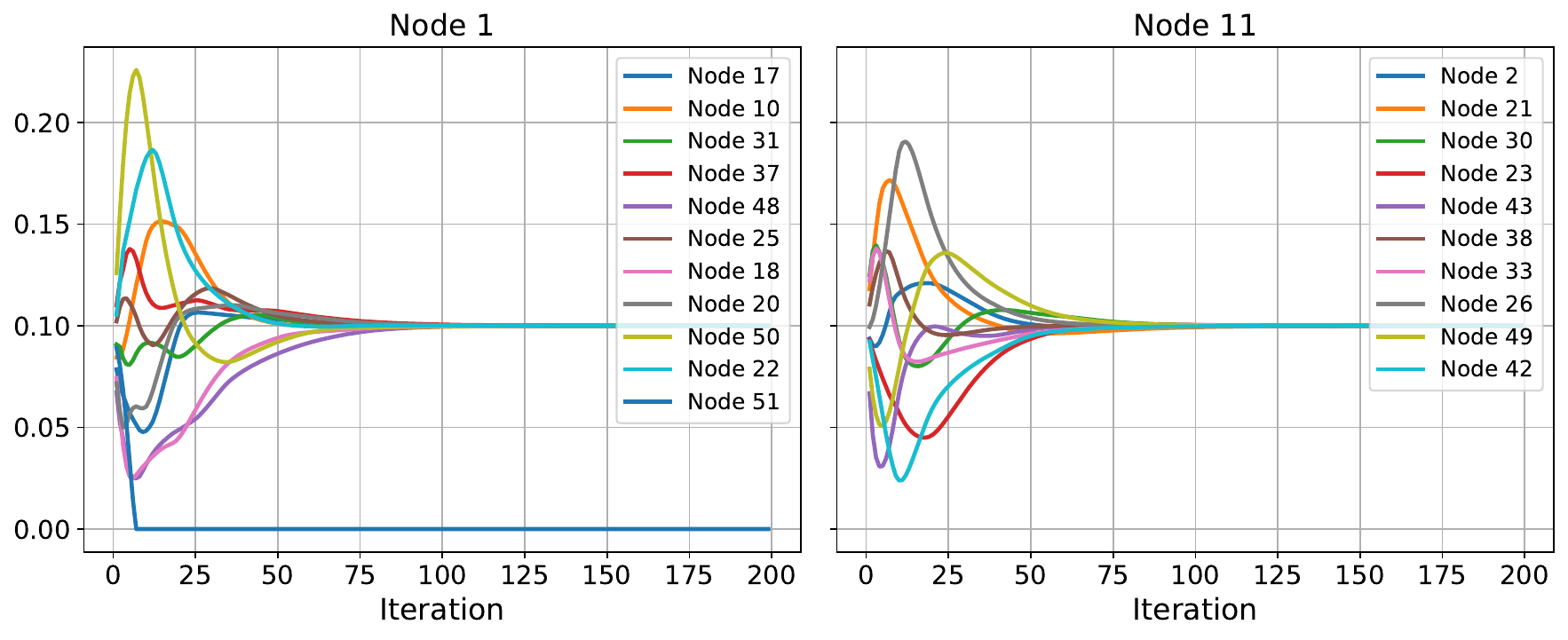}

        (a) A-RepC
    \end{minipage}

    \vspace{1mm}

    \begin{minipage}{\linewidth}
        \centering
        \includegraphics[width=\linewidth]{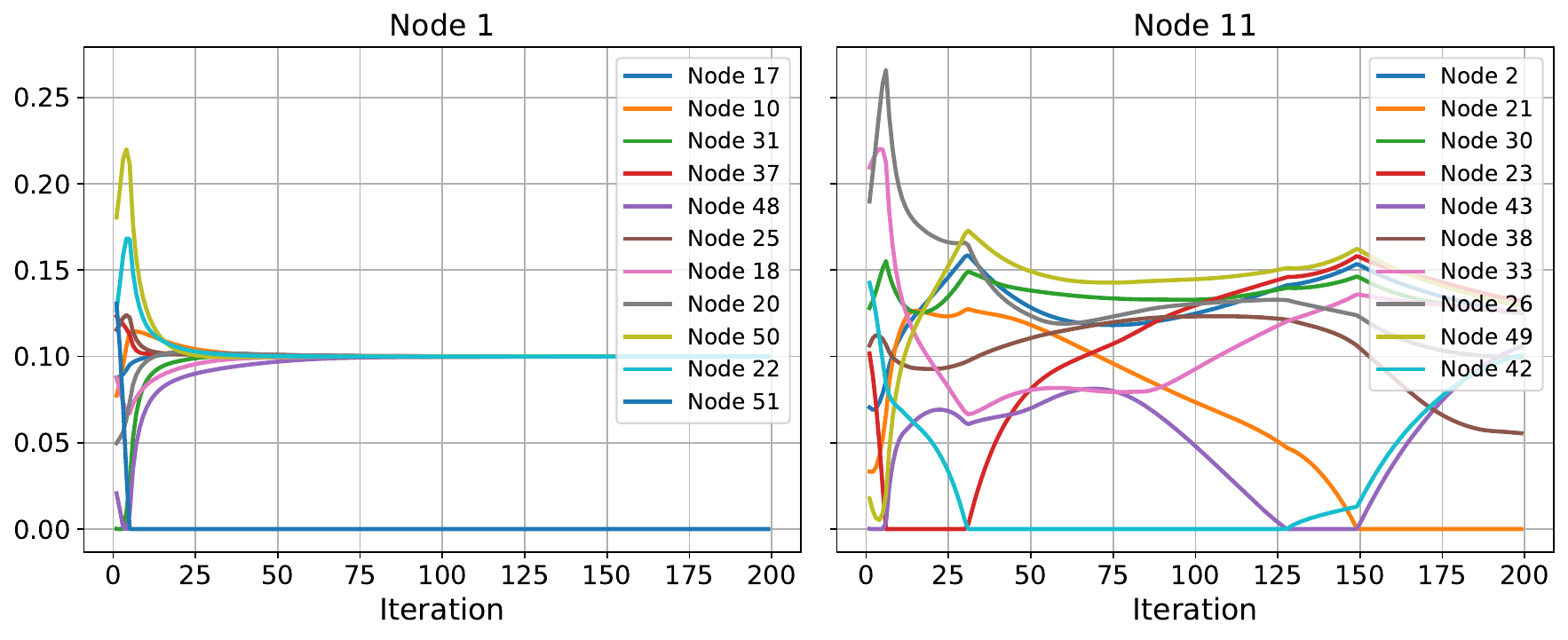}

        (b) RepC
    \end{minipage}

    \vspace{1mm}

    \begin{minipage}{\linewidth}
        \centering
        \includegraphics[width=\linewidth]{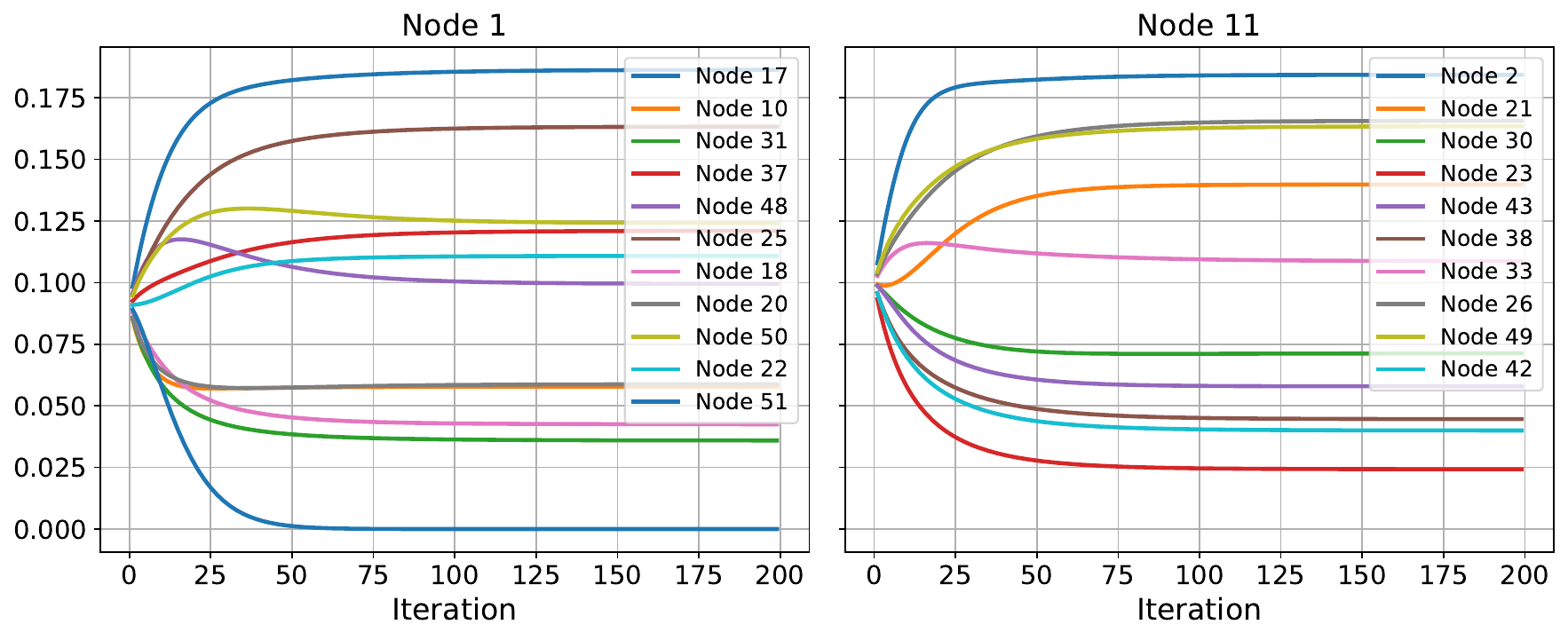}

        (c) WLA
    \end{minipage}

    \caption{Reputation evolutions under A-RepC, RepC, and WLA in the large-scale low-dimensional benchmark with random Byzantine attacks.}
    \label{fig:reputation_evolutions_large_scale}
\end{figure}

\subsection{Robustness under High-Dimensional Mixed Attacks}

We evaluate the applicable methods on a physical Wi-Fi network consisting of $8$ Raspberry Pi devices, one laptop, and one desktop, with node states in $\mathbb{R}^{20}$.
The network topology is shown in Fig.~\ref{fig:small_scale_network_topology}.
Honest nodes initialize their states by sampling uniformly from $[-100,100]^{20}$.

To evaluate robustness under mixed attacks, nodes $8$ and $10$ act as random Byzantine nodes and transmit vectors sampled uniformly from $[-100,100]^{20}$ at every iteration.
For the remaining Byzantine nodes, after receiving neighbor messages $\{x_{j\to i}^{(t)}\}_{j\in\mathcal N_i}$ at iteration $t$, the message sent to neighbor $j$ at the next transmission is
\begin{equation*}
    x_{i\to j}^{(t+1)}
    =
    \begin{cases}
        x_{j\to i}^{(t)} + 100v, & t\bmod 10=0,      \\
        x_{j\to i}^{(t)},        & \text{otherwise}.
    \end{cases}
\end{equation*}
where $v=(1,0,\dots,0)^\top \in \mathbb{R}^{20}$.

\begin{figure}[hbt!]
    \centering
    \includegraphics[width=0.45\textwidth]{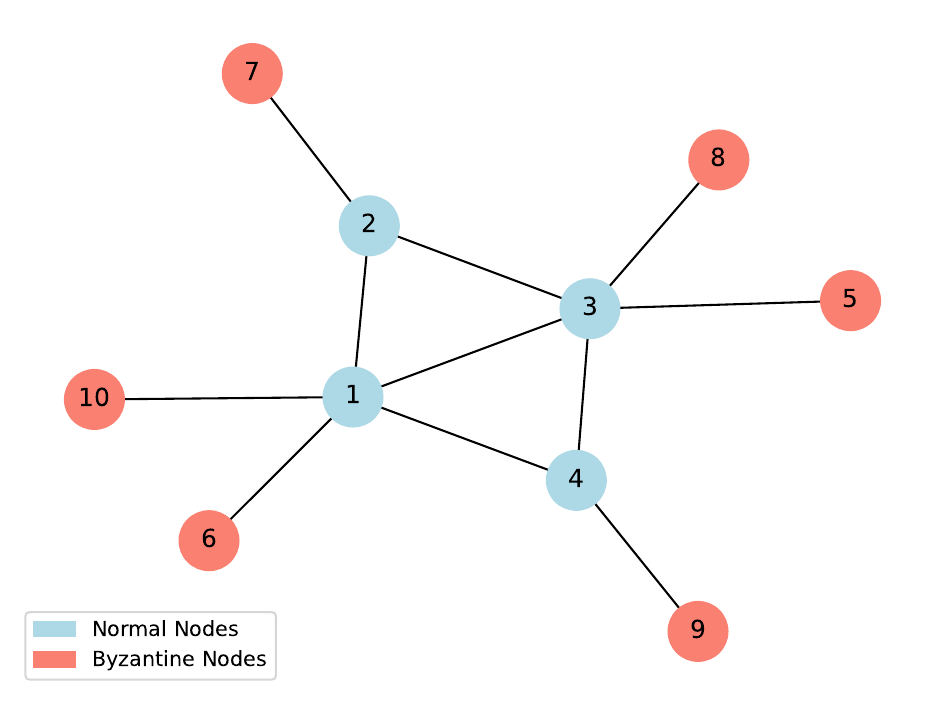}
    \caption{Topology of the physical network used in the small-scale high-dimensional benchmark.}
    \label{fig:small_scale_network_topology}
\end{figure}

ADRC is excluded because centerpoint-based methods are impractical in high dimensions \cite{abbas2022resilient}.
In particular, the two approximate high-dimensional centerpoint algorithms in \cite{abbas2022resilient} require prohibitively large neighborhood sizes: Iterated Tverberg~\cite{MILLER2010647} needs at least $1765$ neighbors for a single iteration, and Iterated Radon~\cite{clarkson1993approximating} requires at least $22$ neighbors.
Since our network does not provide such connectivity, ADRC is infeasible in this setting.

We set $\alpha=0.3$, $\eta=0.002$, and $\lambda=29/30$ for A-RepC; $f=1$ for W-MSR and RepC; $\epsilon=0.001$ for RepC; and $\theta=0.001$ for WLA.

Figure~\ref{fig:rmse_and_dia_small_scale} shows that A-RepC converges steadily to a consensus value that remains close to the initial average.
In contrast, W-MSR, RepC, and WLA all diverge.
For W-MSR and RepC, this is because both methods rely on the preset bound $f=1$, whereas node~$1$ has two Byzantine neighbors, violating this assumption.
WLA fails for a different reason: because its loss function uses the node's own state as the reference, an adversary can readily infer how to construct messages that achieve a high instantaneous score after receiving that state.
This makes WLA particularly vulnerable to periodic relay-based attacks.
Moreover, since WLA leverages the full history of past information, the influence of such strategically crafted messages can accumulate over time, further amplifying the effect of intermittent perturbations and ultimately preventing effective defense against malicious information.

\begin{figure}[hbt!]
    \centering
    \includegraphics[width=1.0\linewidth]{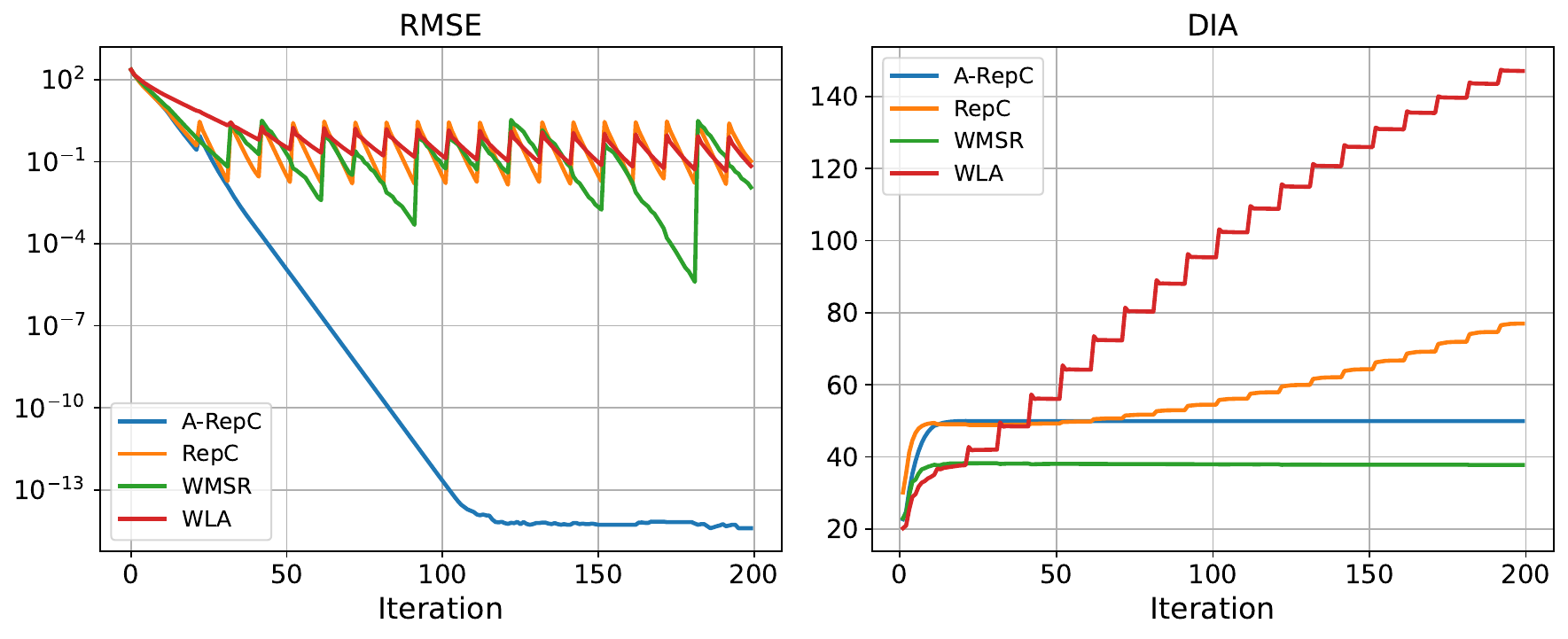}
    \caption{RMSE and DIA evolutions of A-RepC, W-MSR, RepC, and WLA in the small-scale network with mixed Byzantine attacks.}
    \label{fig:rmse_and_dia_small_scale}
\end{figure}

Figure~\ref{fig:reputation_evolutions_small_scale_random} further shows that A-RepC successfully eliminates Byzantine neighbors by assigning them zero weight, while distributing weights approximately uniformly among honest neighbors.
By contrast, RepC performs poorly even for nodes~$2$ and $4$, although their numbers of Byzantine neighbors satisfy the assumed bound $f=1$, because RepC is memoryless and thus cannot accumulate evidence across iterations, making it particularly vulnerable to intermittent attacks.
The weight allocation of WLA is also consistent with the preceding analysis: the intermittently attacking Byzantine nodes~$5$, $6$, $7$, and $9$ receive the largest weights during the initial stage, and these weights decrease only slowly afterward.

\begin{figure}[hbt!]
    \centering

    \begin{minipage}{\linewidth}
        \centering
        \includegraphics[width=\linewidth]{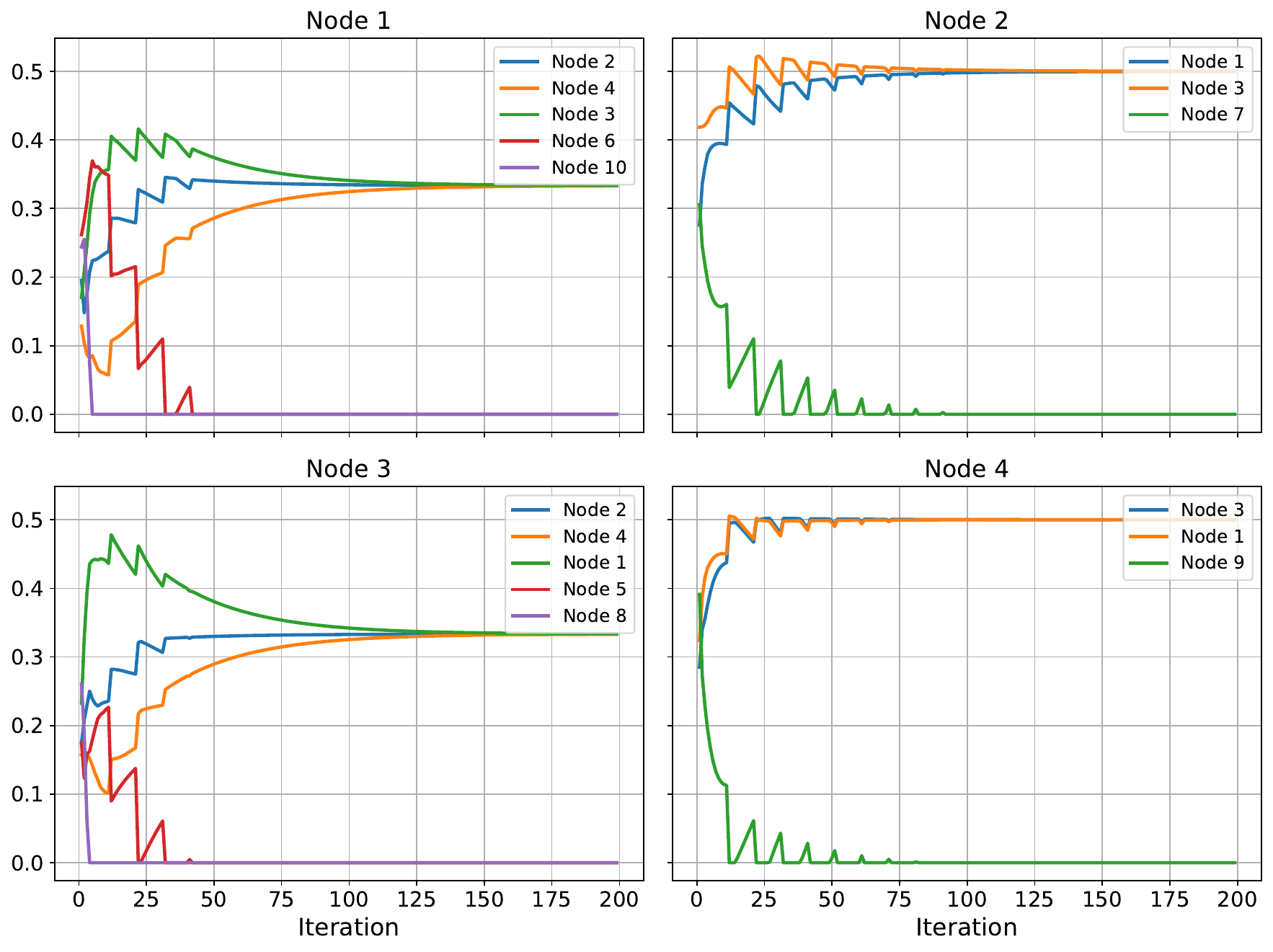}

        (a) A-RepC
    \end{minipage}

    \vspace{1mm}

    \begin{minipage}{\linewidth}
        \centering
        \includegraphics[width=\linewidth]{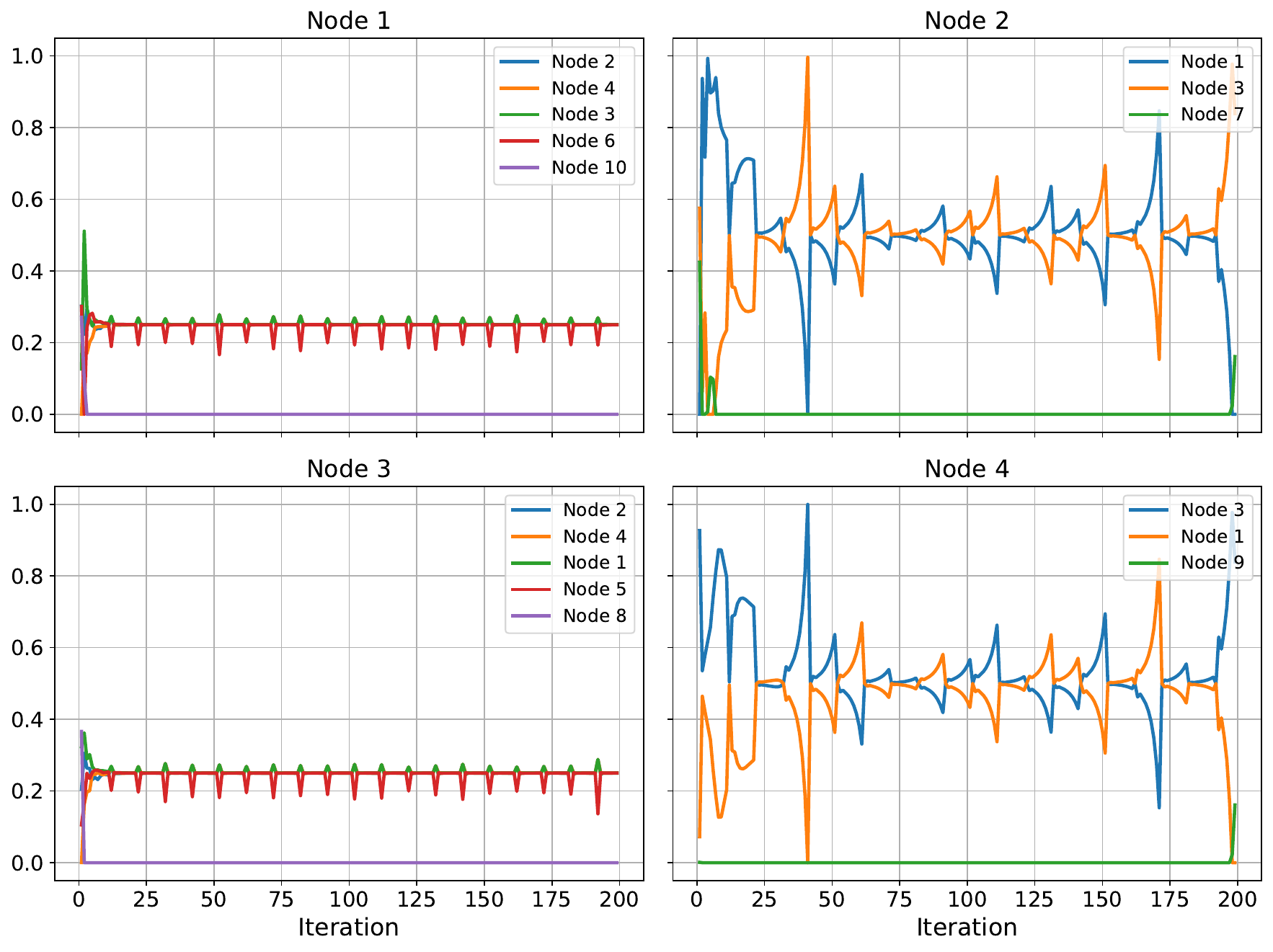}

        (b) RepC
    \end{minipage}

    \vspace{1mm}

    \begin{minipage}{\linewidth}
        \centering
        \includegraphics[width=\linewidth]{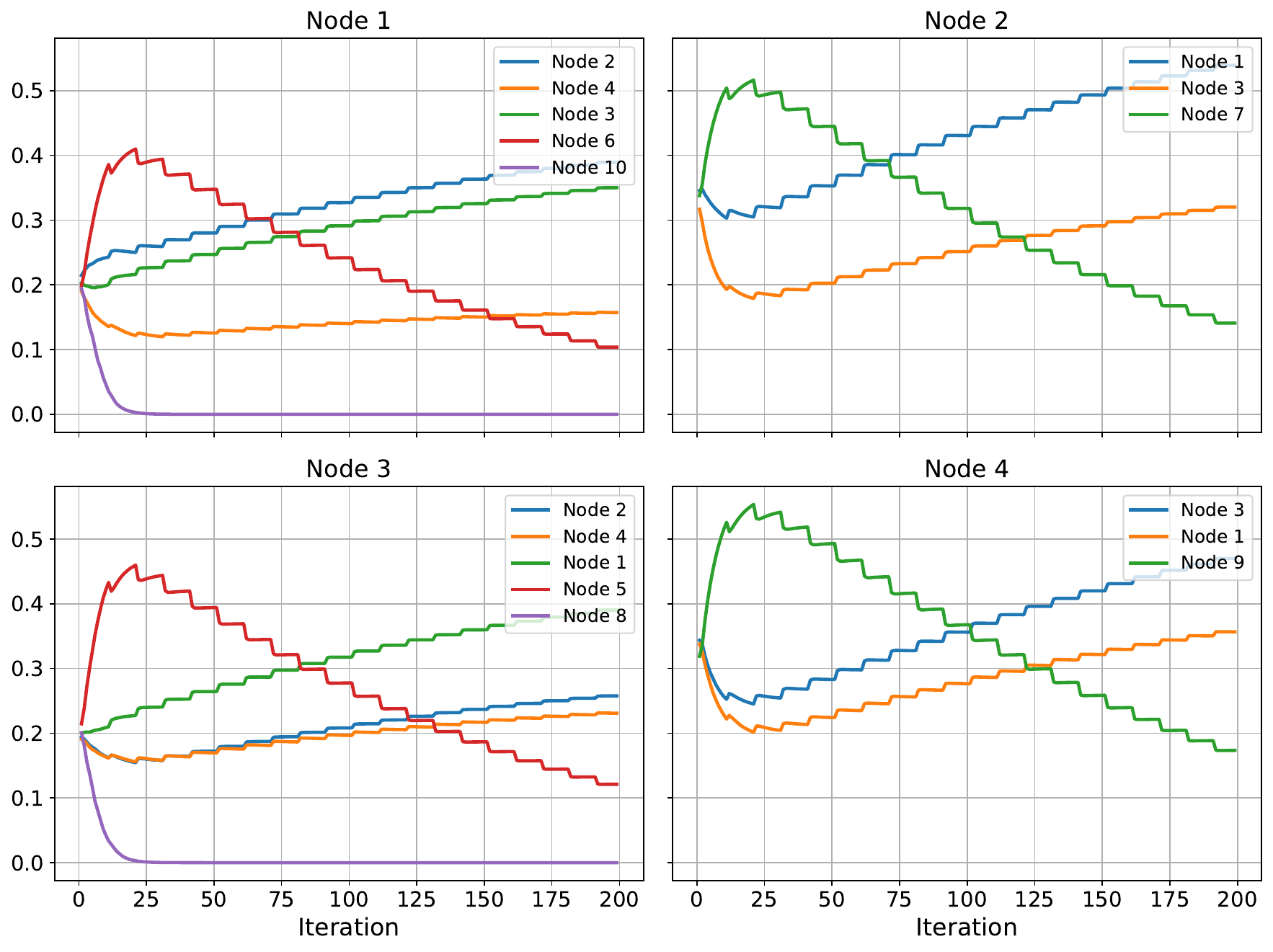}

        (c) WLA
    \end{minipage}

    \caption{Reputation dynamics of A-RepC, RepC, and WLA in the small-scale network with mixed Byzantine attacks.}
    \label{fig:reputation_evolutions_small_scale_random}
\end{figure}

\section{Conclusion}
\label{sec:conclusion}

This paper proposed A-RepC, a Byzantine-resilient consensus algorithm that jointly addresses Byzantine behavior discrimination and resilient consensus among honest agents.
By embedding a learning-inspired reputation mechanism into the consensus loop, A-RepC enables agents to infer neighbor trustworthiness from outlier-robust local losses and historical information, and to use the resulting reputations to suppress adversarial influence in subsequent updates.
Theoretical analysis established rigorous resilience guarantees, and extensive distributed-system experiments showed that A-RepC achieves more effective Byzantine identification and more reliable consensus than existing baselines, especially in high-dimensional settings.

Future work will focus on tightening the theoretical guarantees and extending the analysis to time-varying and directed networks.

\bibliographystyle{IEEEtran}
\bibliography{main}

\end{document}